\documentclass[reqno,10pt,centertags]{amsart}
\usepackage{amsmath,amsthm,amscd,amssymb,latexsym,esint,upref,stmaryrd,
enumerate,color,verbatim,yfonts,mathrsfs} 

\usepackage{hyperref}

\newcommand*{\mailto}[1]{\href{mailto:#1}{\nolinkurl{#1}}}

\newcommand{\bbC}{{\mathbb{C}}}

\newcommand{\bbN}{{\mathbb{N}}}

\newcommand{\bbR}{{\mathbb{R}}}

\newcommand{\cB}{{\mathcal B}}

\newcommand{\cF}{{\mathcal F}}

\newcommand{\cH}{{\mathcal H}}

\newcommand{\cM}{{\mathcal M}}
\newcommand{\cN}{{\mathcal N}}

\newcommand{\cP}{{\mathcal P}}

\newcommand{\cS}{{\mathcal S}}

\newcommand{\gB}{{\mathfrak{B}}}

\DeclareMathOperator{\supp}{supp}

\DeclareMathOperator{\dom}{dom}

\DeclareMathOperator{\tr}{tr}

\renewcommand{\Re}{\text{\rm Re}}
\renewcommand{\Im}{\text{\rm Im}}
\renewcommand{\ln}{\text{\rm ln}}

\newcommand{\loc}{\text{\rm{loc}}}

\newcommand{\no}{\notag}
\newcommand{\lb}{\label}
\newcommand{\f}{\frac}

\newcommand{\ol}{\overline}

\newcommand{\wti}{\widetilde}

\newcommand{\bi}{\bibitem}

\let\geq\geqslant
\let\leq\leqslant

\makeatletter
\def\theequation{\@arabic\c@equation}

\allowdisplaybreaks 
\numberwithin{equation}{section}

\newtheorem{theorem}{Theorem}[section]
\newtheorem{proposition}[theorem]{Proposition}
\newtheorem{lemma}[theorem]{Lemma}
\newtheorem{corollary}[theorem]{Corollary}
\newtheorem{definition}[theorem]{Definition}

\newtheorem{example}[theorem]{Example}

\theoremstyle{remark}
\newtheorem{remark}[theorem]{Remark}

\begin{document}

\numberwithin{equation}{section}
\allowdisplaybreaks

\title[(Conditional) Positive Semidefiniteness]{On (Conditional) Positive Semidefiniteness in a Matrix-Valued Context}

\author[F.\ Gesztesy]{Fritz Gesztesy}
\address{Department of Mathematics,
University of Missouri, Columbia, MO 65211, USA}
\address{Present address: Department of Mathematics, 
Baylor University, One Bear Place \#97328,
Waco, TX 76798-7328, USA}
\email{\mailto{Fritz\_Gesztesy@baylor.edu}}
\urladdr{\url{http://www.baylor.edu/math/index.php?id=935340}}

\author[M.\ Pang]{Michael Pang}
\address{Department of Mathematics,
University of Missouri, Columbia, MO 65211, USA}
\email{\mailto{pangm@missouri.edu}}
\urladdr{\url{https://www.math.missouri.edu/people/pang}}


\thanks{{\it Studia Math.} {\bf 236}, 143--192 (2017).} 

\date{\today}

\begin{abstract} 
In a nutshell, we intend to extend Schoenberg's classical theorem connecting conditionally 
positive semidefinite functions $F\colon \mathbb{R}^n \to \bbC$, $n \in \mathbb{N}$, and their 
positive semidefinite exponentials $\exp(tF)$, $t > 0$, to the case of matrix-valued functions 
$F \colon \mathbb{R}^n \to \mathbb{C}^{m \times m}$, $m \in \mathbb{N}$. Moreover, we study  
the closely associated property that $\exp(t F(- i \nabla))$, $t>0$, is positivity preserving and 
its failure to extend directly in the matrix-valued context. 
\end{abstract}

\maketitle 

{\scriptsize{\tableofcontents}}

\section{Introduction}  \lb{s1}

To set the stage and hence describe the matrix-valued extensions of some of the classical results on 
(conditional) positive semidefiniteness we are interested in, we first briefly recall the basic 
definitions of positive semidefinite and conditionally positive semidefinite matrices 
$A \in \bbC^{m \times m}$ and positive semidefinite and conditionally positive semidefinite 
functions $F \colon \bbR^n \to \bbC$, and then state three classical results in this context:

\begin{definition} \lb{d1.1}
Let $m \in \bbN$, and $A \in \bbC^{m \times m}$, and 
suppose that $F \colon \bbR^n \to \bbC$, $n \in \bbN$. \\[1mm]
$(i)$ $A$ is called {\bf positive semidefinite}, also denoted by $A \geq 0$, if 
\begin{equation}
(c, A c)_{\bbC^m} = \sum_{j,k=1}^m \ol{c_j} \, A_{j,k} c_k \geq 0 \, 
\text{ for all } \, c=(c_1,\dots,c_m)^{\top} \in \bbC^m. 
\end{equation} 
$(ii)$ $A=\{A_{j,k}\}_{1\leq j,k \leq m} = A^* \in \bbC^{m \times m}$ is said to 
be {\bf conditionally positive semidefinite} if 
\begin{equation}
(c, A c)_{\bbC^m} \geq 0 \, 
\text{ for all } \, c=(c_1,\dots,c_m)^{\top} \in \bbC^m, \, \text{ with } \, \sum_{j=1}^m c_j = 0. 
\end{equation} 
$(iii)$ $F$ is called {\bf positive semidefinite} if for all $N \in \bbN$, $x_p \in \bbR^n$, 
$1 \leq p \leq N$, the matrix $\{F(x_p - x_q)\}_{1 \leq p,q \leq N} \in \bbC^{N \times N}$ is 
positive semidefinite. \\[1mm] 
$(iv)$ $F$ is called {\bf conditionally positive 
semidefinite} if for all $N \in \bbN$, $x_p \in \bbR^n$, $1 \leq p \leq N$, the 
matrix $\{F(x_p - x_q)\}_{1 \leq p,q \leq N} \in \bbC^{N \times N}$ is conditionally positive 
semidefinite. \\[1mm]
$(v)$ $F$ is called {\bf positive semidefinite in the sense of Schoenberg} if $F(-x) = \ol{F(x)}$, 
$x \in \bbR^n$, and if for all 
$N \in \bbN$, $x_p \in \bbR^n$, $1 \leq p \leq N$, the matrix 
$\big\{F(x_p - x_q) - F(x_p) - \ol{F(x_q)}\big\}_{1 \leq p,q \leq N} \in \bbC^{N \times N}$ is 
positive semidefinite. \\[1mm] 
$(vi)$ 
Let $T \in \cB\big( L^2(\bbR^n)\big)$.~Then $T$ is called {\bf positivity preserving} 
$($in $ L^2(\bbR^n)$$)$ if for any $0 \leq f \in L^2(\bbR^n)$ also $T f \geq 0$. 
\end{definition}

In connection with Definition \ref{d1.1}\,$(iv)$ one can show that if $F$ is conditionally positive 
semidefinite, then $F(-x) = \ol{F(x)}$, $x \in \bbR^n$. In addition, one observes that for $T$ to be positivity preserving it suffices to take $0 \leq f \in C_0^{\infty} (\bbR^n)$ in Definition \ref{d1.1}\,$(vi)$. 

Given the notions just introduced in Definition \ref{d1.1}, we now recall three classical results. We 
start with Schoenberg's Theorem \cite{Sc38}, who studied isometric imbeddability of 
separable spaces with appropriate distance functions into a Hilbert space.

\begin{theorem} [cf., e.g., {\cite{BCR76}}, {\cite[Sect.\ 3.6]{Ja01}}, {\cite[Proposition~4.4]{SSV12}}] \lb{t1.3} ${}$ \\
Assume that $F \colon \bbR^n \to \bbC$. Then the following conditions $(i)$--$(iii)$ are 
equivalent: \\[1mm] 
$(i)$ $F(0) \leq 0$ and $F$ is conditionally positive semidefinite.
\\[1mm]
$(ii)$ $F(0) \leq 0$ and for all $t > 0$, $\exp(t F)$ is positive semidefinite. \\[1mm]
$(iii)$ $F$ is positive semidefinite in the sense of Schoenberg. \\[1mm] 
If, in addition, $F$ is locally bounded and one of conditions $(i)$--$(iii)$ holds, there exists 
$C > 0$ such that
\begin{equation}
|F(x)| \leq C \big[1 + |x|^2\big], \quad x \in \bbR^n.    \lb{1.7}
\end{equation}
\end{theorem} 

\noindent 
In this context see also \cite[Sects.\ 4.3, 4.4]{BCR84} and \cite[Sect.~II.7]{BF75}. 

Given $F \in C(\bbR^n)$ and $F$ polynomially bounded, one can define 
\begin{equation}
F(- i \nabla) \colon \begin{cases} C_0^{\infty}(\bbR^n) \to L^2(\bbR^n),  \\ 
f \mapsto F(- i \nabla) f = \big(f^{\wedge} F\big)^{\vee}. \end{cases} 
\end{equation}
More generally, if $F \in L^1_{\loc}(\bbR^n)$, one introduces the maximally defined operator of 
multiplication by $F$ in $L^2(\bbR^n)$, denoted by $M_F$, by
\begin{equation}
(M_F f)(x) = F(x) f(x), \quad 
f \in \dom(M_F) = \big\{g \in L^2(\bbR^n) \, \big | \, F g \in L^2(\bbR^n)\big\}, 
\end{equation} 
and then defines $F(- i \nabla)$ as a normal operator in $L^2(\bbR^n)$ via
\begin{equation} 
F(- i \nabla) = \cF^{-1} M_F \cF 
\end{equation}
(cf.\ \eqref{1.15}, \eqref{1.16} and their unitary extensions to $L^2(\bbR^n)$).  

\begin{theorem} [cf., e.g., {\cite{HS78}}, {\cite{JS98}}, {\cite[Theorems~XIII.52 and XIII.53]{RS78}}] \lb{t1.4} 
${}$ \\[1mm] 
Assume that $F \in C(\bbR^n)$ and there exists $c \in \bbR$ such that $\Re(F(x)) \leq c$. Then the following conditions $(i)$--$(iv)$ are equivalent: \\[1mm] 
$(i)$ For all $t > 0$, $\exp(t F(-i \nabla))$ is positivity preserving. \\[1mm] 
$(ii)$ For each $t > 0$, $e^{t F}$ is a positive semidefinite function. \\[1mm] 
$(iii)$ $F(-x) = \ol{F(x)}$, $x \in \bbR^n$, and $F$ is conditionally positive semidefinite.
\\[1mm]
$(iv)$ $($The Levy--Khintchine formula\,$)$. There exists, $\alpha \in \bbR$, $\beta \in \bbR^n$, 
$0 \leq A \in \bbC^{n \times n}$, and a nonnegative finite measure $\nu$ on $\bbR^n$, with 
$\nu(\{0\}) = 0$, such that 
\begin{align}
\begin{split} 
F(x) &= \alpha + i (\beta \cdot x) - (x \cdot (A x))   \\ 
& \quad + \int_{\bbR^n} \bigg[\exp(i (x \cdot y)) -1 - 
\frac{i (x \cdot y)}{1 + |y|^2} \bigg] \f{1 + |y|^2}{|y|^2} \, d \nu(y),  
\quad x \in \bbR^n.
\end{split} 
\end{align}
\end{theorem}
 
The principal aim of this paper is to investigate to which degree Theorem \ref{t1.3} and 
Theorem \ref{t1.4}\,$(i)$--$(iii)$ extend to the matrix-valued context, where 
$F \colon \bbR^n \to \bbC^{m \times m}$, $m \in \bbN$, $m \geq 2$, and, if direct generalizations are impossible, in what modified form do extensions exist. We also note that a matrix-valued extension 
of  the Levy--Khintchine formula, Theorem \ref{t1.4}\,$(iv)$, while not the subject of this paper, is 
part of ongoing investigations. For a historical survey on infinitely divisible distributions and their connection to the Levy--Khintchine formula we refer to \cite{MR06} (and the extensive list of 
references cited therein). 

For completeness we also recall Bochner's theorem \cite{Bo33} as it naturally fits in with Theorems \ref{t1.3} and \ref{t1.4}:

\begin{theorem} [Bochner's Theorem, cf., e.g., {\cite[Sect.~5.4]{Ak65}}, 
{\cite[p.~13]{RS75}}, {\cite[p.~46]{SSV12}}] \lb{t1.2} 
${}$ \\ 
Assume that $F \in C(\bbR^n)$. Then the following conditions $(i)$ and 
$(ii)$ are equivalent: \\[1mm] 
$(i)$ $F$ is positive semidefinite. \\[1mm]
$(ii)$ There exists a nonnegative finite measure $\mu$ on $\bbR^n$ such that 
\begin{equation}
 F(x) = \mu^{\wedge}(x), \quad x \in \bbR^n.    \lb{1.3} 
\end{equation} 
In addition, if one of conditions $(i)$ or $(ii)$ holds, then
\begin{equation}
F(-x) = \ol{F(x)}, \quad |F(x)| \leq |F(0)|, \quad x \in \bbR^n, 
\end{equation} 
in particular, $F$ is bounded on $\bbR^n$. 
\end{theorem}

In this context we emphasize that the extension of Bochner's Theorem \ref{t1.2} has been obtained 
by Berberian \cite{Be66} not only in the matrix context (cf.\ Theorem \ref{t4.2}), but in the 
infinite-dimensional case in connection with Abelian groups. As a result, we exclusively focus on extensions of Theorem \ref{t1.3} and Theorem \ref{t1.4}\,$(i)$--$(iii)$ in the following. 

Turning to the matrix-valued case, $F \colon \bbR^n \to \bbC^{m \times m}$, $m \in \bbN$, 
and taking the notions of positive semidefinite and conditionally positive semidefinite matrix-valued 
functions $F$ in Definition \ref{d2.4} (and the obvious matrix-valued extension of 
Definition \ref{d1.1}\,$(v)$) for granted, we can now briefly describe the form in which 
Theorem \ref{t1.3} and Theorem \ref{t1.4}\,$(i)$--$(iii)$ extend to the matrix-valued context: 
First, and foremost, \\[1mm] 
\indent 
$\bullet$ the exponential $\exp(t F)$ must consistently be replaced by the Hadamard ex- 
\hspace*{6mm} ponential $\exp_H(tF)$ in the matrix context. \\[1mm]
Here the {\it Hadamard exponential} $\exp_H(G(x))$ of $G \colon \bbR^n \to \bbC^{m \times m}$, 
$m \in \bbN$, is defined by 
\begin{equation}
\exp_H(G(x)) = \big\{\exp_H(G(x))_{j,k} := \exp(G(x)_{j,k})\big\}_{1 \leq j,k \leq m}, \quad x \in \bbR^n.  
\end{equation}
It is understood in the following that $\exp(t F)$ is always replaced by the Hadamard exponential 
$\exp_H(tF)$ in the matrix context $m \in \bbN$, $m \geq 2$.

In connection with the matrix-valued extension of Schoenberg's Theorem \ref{t1.3} (for $m \in \bbN$, 
$m \geq 2$) we prove the following facts in Theorem \ref{t4.5a} and Remark \ref{r4.5b}: \\[1mm] 
\indent 
$\bullet$ Items $(i)$ and $(ii)$ in Theorem \ref{t1.3} remain equivalent (disregarding the condition 
\hspace*{6mm} $F(0) \leq 0$). \\[1mm]
\indent 
$\bullet$ If $F(0) \leq 0$ and one of conditions $(i)$ and $(ii)$ in Theorem \ref{t1.3} holds, then 
\hspace*{6mm} condition $(iii)$ in Theorem \ref{t1.3} is implied, but we prove that the converse is 
\hspace*{6mm} false in the matrix-valued context. 

\vspace*{1mm} 

In connection with the matrix-valued extension of Theorem \ref{t1.4} (for $m \in \bbN$, 
$m \geq 2$) we prove the following facts in Theorems \ref{t4.6} and \ref{t4.12}: \\[1mm] 
\indent 
$\bullet$ Conditions $(ii)$ and $(iii)$ of Theorem \ref{t1.4} remain equivalent in the matrix-valued  \\
 \hspace*{6mm} context, however, item $(i)$ does not extend at all (employing $\exp_H(tF)$ 
as agreed \hspace*{6mm} upon).  We did find a proper extension of condition $(i)$ (cf.\ 
Theorem \ref{t4.6}\,$(i)$). \\[1mm] 

These comments illustrate that much of Theorems \ref{t1.3} and \ref{t1.4} extends to the 
matrix-valued context, but some items require very specific modifications. In particular, the positivity 
preserving condition $(i)$ in Theorem \ref{t1.4} needs to be altered sharply. 

Next, we briefly turn to the contents of each section. Section \ref{s2} is of preparatory nature and 
recalls the basic facts on positive semidefinite and conditionally positive semidefinite matrices and 
matrix-valued functions on $\bbR^n$, $n \in \bbN$, introduces the notion of the Hadamard 
exponential, and derives the equivalence of items $(i)$ and $(ii)$ in Schoenberg's Theorem \ref{t1.3} 
in the matrix-valued context. Introductory remarks on convolution operators involving matrix-valued measures are the contents of Section \ref{s3}. We recall the spaces 
$L^p(\bbR^n, \bbC^{m \times m})$, $p \in [1,\infty) \cup \{\infty\}$, discuss the operator 
$F(- i \nabla)$, $F \in L^{\infty}(\bbR^n, \bbC^{m \times m})$, via Fourier transform, discuss 
various consequences of positivity preserving of $F(- i \nabla)$, and conclude this section with two 
approximation results (cf.\ Lemmas \ref{l3.12} and \ref{l3.11}). Our principal results are formulated 
in Section \ref{s4}. The classical $L^1$ and $L^2$ Fourier multiplier results are discussed in the 
matrix-valued context in Theorems \ref{t4.3} and \ref{t4.4}. The matrix-valued extension of 
Schoenberg's Theorem \ref{t1.3} is formulated in Theorem \ref{t4.5a}; the fact that no complete 
extension of Theorem \ref{t1.3} is possible (in the sense that either of conditions $(i)$ and $(ii)$ of 
Theorem \ref{t1.3} implies its condition $(iii)$, but that the converse is false) is demonstrated in 
Remark \ref{r4.5b}. The extent to which Theorem \ref{t1.4} extends to the matrix-valued case is 
dealt with in detail in Theorems \ref{t4.6} and \ref{t4.12}, as well as Remark \ref{r4.7}. The analog 
of the bound \eqref{1.7} in the matrix-valued context is derived in Theorem \ref{t4.14}. 
Appendix \ref{sA} constructs a counterexample verifying the claim made in Remark \ref{r4.1a}, and 
Appendix \ref{sB} provides a proof of \eqref{4.41A}. 

\smallskip

Finally, we briefly summarize the basic notation employed in this paper: Let $\cH$ be a separable complex Hilbert space, $(\, \cdot \,,\, \cdot \,)_{\cH}$ the scalar product in $\cH$ (linear in the second argument), and $I_{\cH}$ the identity operator in $\cH$.

The Banach spaces of bounded and compact linear operators on a separable complex Hilbert 
space $\cH$ are denoted by $\cB(\cH)$ and $\cB_\infty(\cH)$, respectively; the corresponding 
$\ell^p$-based Schatten--von Neumann trace ideals (cf.\ \cite[Ch.~III]{GK69}, \cite[Ch.~1]{Si05}) 
will be denoted by $\cB_p (\cH)$, with corresponding norm denoted by  
$\|\, \cdot \,\|_{\cB_p(\cH)}$, $p \geq 1$ (and defined in terms of the $\ell^p$-norm of the singular values 
of the operator in question). Moreover, $\tr_{\cH}(A)$ denotes the trace of a trace class operator 
$A\in\cB_1(\cH)$. We also employ the analogous notation $\cB(X_1, X_2)$ for bounded linear 
operators mapping the Banach space $X_1$ into the Banach space $X_2$. 

For $X$ a set, $X^{m \times n}$, $m,n \in \bbN$, represents the set of 
$m \times n$ matrices with entries in $X$.   

Unless explicitly stated otherwise, $\bbC^m$ is always equipped with the Euclidean scalar 
product $( \, \cdot \, , \, \cdot \,)_{\bbC^m}$ and associated norm $\|\, \cdot \, \|_{\bbC^m}$.  

For $A \in \bbC^{m \times m}$, $m \in \bbN$, we denote by $A^{\top}$ the transpose of $A$, and 
by $\|A\|_{\cB(\bbC^m)}$ the operator norm of $A$, considering $A$ as a linear operator on 
$\bbC^m$ (equipped with $\|\, \cdot \,\|_{\bbC^m}$). In this context we note that
\begin{equation}
\big(\bbC^{m \times m}, \| \cdot \|_{\cB(\bbC^m)}\big)^*  
= \big(\bbC^{m \times m}, \| \cdot \|_{\cB_1(\bbC^m)}\big).    \lb{1.8} 
\end{equation}
We also introduce  
\begin{equation}
\|A\|_{\max} = \max_{1 \leq j, k \leq m} |A_{j,k}|.    \lb{1.9} 
\end{equation}

The symbol $\cS(\bbR^n, \bbC^{m \times m})$ denotes the space of all $\bbC^{m \times m}$-valued 
rapidly decreasing functions on $\bbR^n$ with each entry in the usual Schwartz 
space $\cS(\bbR^n)$. In addition, we introduce the spaces,
\begin{align} 
& C_0(\bbR^n, \bbC^{m \times m}) = \{f \in C(\bbR^n, \bbC^{m \times m}) \, | \, \supp\,(f) \, \text{compact}\},    \lb{1.10} \\
& C_{b}(\bbR^n, \bbC^{m \times m}) = \{f \in C(\bbR^n, \bbC^{m \times m}) \, | \, \|f\|_{\infty} < \infty\}, 
\lb{1.11} \\
& C_{\infty}(\bbR^n, \bbC^{m \times m}) = \big\{f = \{f_{j,k}\}_{1 \leq j,k \leq m} \colon 
\bbR^n \to \bbC^{m \times m} \, \big| \, f_{j,k} \in C(\bbR^n),    \no \\ 
& \hspace*{5.55cm} \lim_{|x| \to \infty} f_{j,k}(x) = 0, \, 1 \leq j,k \leq m\big\}. 
\lb{1.12}
\end{align}
Unless explicitly stated otherwise, the spaces \eqref{1.10}--\eqref{1.12} are always equipped with the 
norm $\|f\|_{\infty} = {\rm ess.sup}_{x \in \bbR^n} \| f(x) \|_{\cB(\bbC^m)}$. 

For brevity, we will omit displaying the Lebesgue measure $d^n x$ in 
$L^p(\bbR^n, \bbC^{m \times m})$, $p \in [1,\infty)\cup \{\infty\}$, whenever the latter is understood. 

The Fourier and inverse Fourier transforms on $\cS(\bbR^n, \bbC^{m \times m})$ are denoted by the pair of formulas,
\begin{align}
(\cF f)(y) = f^{\wedge}(y) &= (2 \pi)^{-n/2} \int_{\bbR^n} e^{- i (y \cdot x)} f(x) \, d^n x,    \lb{1.15} \\
(\cF^{-1} g)(x) = g^{\vee}(x) &= (2 \pi)^{-n/2} \int_{\bbR^n} e^{i (x \cdot y)} g(y) \, d^n y, \lb{1.16} \\
& \hspace*{1.6cm}  f, g \in \cS(\bbR^n, \bbC^{m \times m}),    \no 
\end{align}
and we use the same notation for the appropriate extensions, where 
$\cS(\bbR^n, \bbC^{m \times m})$ is replaced by $L^1(\bbR^n, \bbC^{m \times m})$ if 
$f \in L^1(\bbR^n, \bbC^{m \times m})$, or by its unitary extension to  
$L^2(\bbR^n, \bbC^{m \times m})$ if $f \in L^2(\bbR^n, \bbC^{m \times m})$. 

The open ball in $\bbR^n$ with center $x_0 \in \bbR^n$ and radius $r_0 > 0$ is denoted by 
the symbol $B_n(x_0, r_0)$, the norm of vectors $x \in \bbR^n$ is denoted by $|x|$, the 
scalar product of $x, y \in \bbR^n$, is abbreviated by $x \cdot y$.  

We denote by $\gB_n$ the $\sigma$-algebra of all Borel subsets of $\bbR^n$ and for $E \in \gB_n$, 
abbreviate the $n$-dimensional Lebesgues measure of $E$ by $|E|$.

\section{Matrix-valued (Conditional) Positive Semidefinite Functions: A Variant of Schoenberg's Theorem} \lb{s2}

In this preparatory section we recall the basic facts on positive semidefinite and conditionally 
positive semidefinite matrices and matrix-valued functions on $\bbR^n$, $n \in \bbN$, introduce the notion of the Hadamard exponential, and derive the equivalence of items $(i)$ and $(ii)$ in Schoenberg's Theorem \ref{t1.3} (see, e.g., \cite{BCR76}, \cite[Sect.\ 3.6]{Ja01}, and \cite[Proposition~4.4]{SSV12}) in the matrix-valued context.  

We start with the following definition (cf., e.g., \cite[p.~180]{Bh07}, \cite[p.~451]{HJ94}).)

\begin{definition} \lb{d2.1}
Let $m \in \bbN$, and $A=\{A_{j,k}\}_{1\leq j,k \leq m} \in \bbC^{m \times m}$. \\[1mm]
$(i)$ $A$ is called {\bf positive semidefinite}, also denoted by $A \geq 0$, if 
\begin{equation}
(c, A c)_{\bbC^m} = \sum_{j,k=1}^m \ol{c_j} \, A_{j,k} c_k \geq 0 \, 
\text{ for all } \, c=(c_1,\dots,c_m)^{\top} \in \bbC^m. 
\end{equation} 
$(ii)$ $A=\{A_{j,k}\}_{1\leq j,k \leq m} = A^* \in \bbC^{m \times m}$ is said to 
be {\bf conditionally positive semidefinite} if 
\begin{equation}
(c, A c)_{\bbC^m} \geq 0 \, 
\text{ for all } \, c=(c_1,\dots,c_m)^{\top} \in \bbC^m, \, \text{ with } \, \sum_{j=1}^m c_j = 0. 
\end{equation} 
\end{definition}

Given $S \in \bbC^{m \times m}$, $m \in \bbN$, its {\bf Hadamard exponential}, denoted 
by $\exp_H(S)$, is defined by 
\begin{equation}
\exp_H(S) = \big\{\exp_H(S)_{j,k} := \exp(S_{j,k})\big\}_{1 \leq j,k \leq m}.  
\end{equation}

\begin{lemma} [see, e.g., {\cite[Theorem~6.3.6]{HJ94}}] \lb{l2.2} 
${}$ \\
Let $A \in \bbC^{m \times m}$, $m \in \bbN$, be conditionally positive semidefinite. Then 
$\exp_H(A) \geq 0$, that is, the Hadamard exponential of $A$ is positive semidefinite. 
\end{lemma}

The following result represents a complexified version of 
\cite[Exercise~5.6.15]{Bh07}, \cite[Theorem~6.3.13]{HJ94}:

\begin{lemma} \lb{l2.3}
Let $\varepsilon > 0$, assume $A = A^* \in \bbC^{m \times m}$, $m \in \bbN$, and suppose that 
$\exp_H(tA)$ is positive semidefinite for all $t \in (0, \varepsilon)$. Then $A$ is conditionally 
positive semidefinite.  
\end{lemma}
\begin{proof}
Let $c = (c_1,\dots,c_m)^{\top} \in \bbC^m$ with $\sum_{j=1}^m c_j =0$. Then for all 
$t \in (0, \varepsilon)$, 
\begin{align}
0 & \leq t^{-1} (c, \exp_H(t A) c)_{\bbC^m} 
= \sum_{j,k=1}^m \ol{c_j} \, t^{-1} \big[\exp_H(t A_{j,k}) -1\big] c_k    \no \\
& \, \underset{t \downarrow 0}{\longrightarrow} \sum_{j,k=1}^m \ol{c_j} \, A_{j,k} c_k  
= (c, A c)_{\bbC^m}. 
\end{align}
\end{proof}

Combining Lemmas \ref{l2.2} and \ref{l2.3} shows that for $A = A^* \in \bbC^{m \times m}$, 
$m \in \bbN$,
\begin{align}
\begin{split}
& \exp_H(tA) \geq 0 \, \text{ for all} \, t \in (0,\varepsilon_0) \, \text{ for some fixed $\varepsilon_0 > 0$} \\
& \quad \text{is equivalent to } \,  \exp_H(tA) \geq 0 \, \text{ for all $t \geq 0$.} 
\end{split}
\end{align}

\begin{definition} \lb{d2.4}
Let $F \colon \bbR^n \to \bbC^{m \times m}$, $m, n \in \bbN$. \\[1mm] 
$(i)$ $F$ is called {\bf positive semidefinite} if for all $N \in \bbN$, $x_p \in \bbR^n$, 
$1 \leq p \leq N$, the block matrix $\{F(x_p - x_q)\}_{1 \leq p,q \leq N} \in \bbC^{mN \times mN}$ is 
positive semidefinite. \\[1mm] 
$(ii)$ $F$ is called {\bf conditionally positive 
semidefinite} if for all $N \in \bbN$, $x_p \in \bbR^n$, $1 \leq p \leq N$, the block 
matrix $\{F(x_p - x_q)\}_{1 \leq p,q \leq N} \in \bbC^{mN \times mN}$ is conditionally positive semidefinite.
\end{definition}

\begin{lemma} \lb{l2.5}
Let $F \colon \bbR^n \to \bbC^{m \times m}$, $m, n \in \bbN$. \\[1mm] 
$(i)$ One verifies that $F \colon \bbR^n \to \bbC^{m \times m}$ is positive semidefinite 
if and only if for all $N \in \bbN$, $x_p \in \bbR^n$, $c_p \in \bbC^m$, $1 \leq p \leq N$, one has 
\begin{equation}
\sum_{p,q = 1}^N (c_p, F(x_p - x_q) c_q)_{\bbC^m} \geq 0.
\end{equation}
$(ii)$ As proved in \cite[p.~178]{Be66}, $F \colon \bbR^n \to \bbC^{m \times m}$ is positive semidefinite 
if and only if for all $N \in \bbN$, $x_p \in \bbR^n$, $c_p \in \bbC$, $1 \leq p \leq N$, 
$f = (f_1,\dots,f_m)^\top \in \bbC^m$, 
\begin{equation}
\sum_{p,q=1}^N \ol{c_p} \, (f, F(x_p - x_q) f)_{\bbC^m} c_q 
= \sum_{p,q=1}^N \sum_{j,k=1}^m \ol{c_p} \, \ol{f_j} \, F(x_p - x_q)_{j,k} f_k c_q \geq 0. 
\end{equation}
$(iii)$ One verifies that $F \colon \bbR^n \to \bbC^{m \times m}$ is conditionally positive semidefinite 
if and only if the following conditions $(\alpha)$ and $(\beta)$ hold: \\[1mm]
$(\alpha)$ $F(-x) = F(x)^*$, $x \in \bbR^n$. \\[1mm]
$(\beta)$ For all $N \in \bbN$, $x_p \in \bbR^n$, $c_p = (c_{p,1},\dots,c_{p,m}) \in \bbC^m$, 
$1 \leq p \leq N$, satisfying 
\begin{equation}
\sum_{p=1}^N \sum_{j=1}^m c_{p,j} = 0,
\end{equation}
one has 
\begin{equation}
\sum_{p,q = 1}^N (c_p, F(x_p - x_q) c_q)_{\bbC^m} \geq 0.
\end{equation}
In addition, one observes that $F \colon \bbR^n \to \bbC^{m \times m}$ satisfies condition 
$(\alpha)$ if and only if it satisfies the following condition $(\alpha')$,  \\[1mm]
$(\alpha')$ For all $N \in \bbN$, $x_p \in \bbR^n$, $1 \leq p \leq N$, the block matrix 
$\{F(x_p - x_q)\}_{1 \leq p,q \leq N} \in \bbC^{mN \times mN}$ is self-adjoint in $\bbC^{mN}$. 
\end{lemma}

Given $S \colon \bbR^n \to \bbC^{M \times M}$, $M,n \in \bbN$, its {\bf Hadamard exponential}, denoted by $\exp_H(S)$, is defined by 
\begin{equation}
\exp_H(S(x)) = \big\{\exp_H(S(x))_{j,k} := \exp(S(x)_{j,k})\big\}_{1 \leq j,k \leq M}, \quad x \in \bbR^n.  
\end{equation}

The next two theorems represent a matrix generalization of a variant of Schoenberg's theorem 
(cf., e.g., \cite[Proposition~4.4]{SSV12}), namely, the equivalence of items $(i)$ and $(ii)$ in 
Theorem \ref{t1.3}, the principal results of this section:

\begin{theorem} \lb{t2.6} 
Let $F \colon \bbR^n \to \bbC^{m \times m}$, $m, n \in \bbN$, be conditionally 
positive semidefinite. Then $\exp_H(F)$ is positive semidefinite. 
\end{theorem}
\begin{proof}
For all $N \in \bbN$, $x_p \in \bbR^n$, $1 \leq p \leq N$, the block matrix 
$\{F(x_p - x_q)\}_{1 \leq p,q \leq N} \in \bbC^{mN \times mN}$ is conditionally positive semidefinite. 
Thus, by Lemma \ref{l2.2}, the block matrix $\exp_H\big(\{F(x_p - x_q)_{}\}_{1 \leq p,q \leq N}\big) \in \bbC^{mN \times mN}$ 
is positive semidefinite. Since 
\begin{equation}
\exp_H\big(\{F(x_p - x_q)\}_{1 \leq p,q \leq N}\big) 
= \big\{\exp_H(F(x_p - x_q))\big\}_{1 \leq p,q \leq N}, 
\end{equation}
this completes the proof. 
\end{proof}

\begin{theorem} \lb{t2.7} 
Let $\varepsilon > 0$, $F \colon \bbR^n \to \bbC^{m \times m}$, and suppose that 
$\exp_H(t F) \colon \bbR^n \to \bbC^{m \times m}$ is positive semidefinite for all 
$t \in (0, \varepsilon)$.~Then $F$ is conditionally positive semidefinite. 
\end{theorem}
\begin{proof} 
Let $N \in \bbN$, $x_p \in \bbR^n$, $1 \leq p \leq N$, and assume that 
$c_p =(c_{p,1},\dots,c_{p,m}) \in \bbC^m$, $1 \leq p \leq N$ satisfy 
\begin{equation}
\sum_{p=1}^N \sum_{j=1}^m c_{p,j} = 0.
\end{equation}
Then for all $t \in (0, \varepsilon)$, Lemma \ref{l2.5}\,$(i)$ yields 
\begin{align}
0 & \leq t^{-1} \sum_{p,q=1}^N (c_p, \exp_H(t F(x_p - x_q)) c_q)_{\bbC^m}   \no \\
& = \sum_{p,q=1}^N \sum_{j,k=1}^m \ol{c_{p,j}} \, t^{-1} \big[\exp(t F(x_p - x_q)_{j,k}) - 1\big] c_{q,k} 
\no \\
& \, \underset{t \downarrow 0}{\longrightarrow} \sum_{p,q=1}^N \sum_{j,k=1}^m 
\ol{c_{p,j}} \, F(x_p - x_q)_{j,k} c_{q,k} = \sum_{p,q=1}^N (c_p, F(x_p - x_q) c_q)_{\bbC^m}. 
\end{align}
By Lemma \ref{l2.5}\,$(iii)$, it remains to show that 
\begin{equation} 
F(-x) = F(x)^*, \quad x \in \bbR^n.   \lb{2.13}
\end{equation}  
To this end one observes that the block matrix 
\begin{equation}
\begin{pmatrix} \exp_H(t F(0)) & \exp_H(t F(x)) \\ \exp_H(t F(-x)) & \exp_H(t F(0)) 
\end{pmatrix} \in \bbC^{2m \times 2m} 
\end{equation}
is positive semidefinite and hence self-adjoint. Thus,
\begin{equation}
\exp_H(t F(-x)) = [\exp_H(t F(x))]^*, \quad x \in \bbR^n, \; t \in (0, \varepsilon).
\end{equation}
Next, let $E_{2m} \in \bbC^{2m \times 2m}$ be the matrix all of whose entries equal $1$. Then  
\begin{align}
t^{-1} \big[\exp_H(t F(-x)) - E_{2m}\big] = t^{-1} \big\{[\exp_H(t F(x))]^* - E_{2m}\big\}, 
\quad x \in \bbR^n, \; t \in (0, \varepsilon),    \lb{2.16} 
\end{align}
and letting $t \downarrow 0$ in \eqref{2.16}, one obtains
\begin{equation}
F(-x)_{j,k} = \ol{F(x)_{k,j}}, \quad x \in \bbR^n, \; 1 \leq j,k \leq m,
\end{equation}
proving \eqref{2.13}.
\end{proof}

Combining Theorems \ref{t2.6} and \ref{t2.7} shows that for $F \colon \bbR^n \to \bbC^{m \times m}$,
\begin{align}
\begin{split}
& \exp_H(t F) \geq 0 \, \text{ for all} \, t \in (0,\varepsilon_0) \, \text{ for some fixed $\varepsilon_0 > 0$} \\
& \quad \text{is equivalent to } \,  \exp_H(t F) \geq 0 \, \text{ for all $t \geq 0$.} 
\end{split}
\end{align}

Next, we intend to show that Definitions \ref{d2.1} and \ref{d2.4} are compatible. 

\begin{corollary} \lb{c2.10}
Let $0 \leq A \in \bbC^{m \times m}$ $($i.e., $A$ is positive semidefinite\,$)$ and introduce 
$F \colon \bbR^n \to \bbC^{m \times m}$ by 
\begin{equation}
F(x) = A, \quad x \in \bbR^n.
\end{equation} 
Then $F \geq 0$, that is, $F$ is positive semidefinite in the sense of Definition \ref{d2.4}\,$(i)$.
\end{corollary}
\begin{proof}
For any $c = (c_1, \dots , c_N)^{\top} \in \bbC^N$, 
\begin{equation}
\sum_{p,q=1}^N \ol{c_p} (f,Af)_{\bbC^m} c_q = 
(f,Af)_{\bbC^m} \sum_{p,q=1}^N \ol{c_p} c_q = 
(f,Af)_{\bbC^m} (c, H_N c)_{\bbC^N},
\end{equation}
where $H_N$ denotes the $N \times N$-matrix with all entries equal to $1$. Since it is 
well-known that $H_N$ is positive semidefinite, 
\begin{equation}
\sum_{p,q=1}^N \ol{c_p} (f,Af)_{\bbC^m} c_q \geq 0. 
\end{equation}   
Thus, Lemma \ref{l2.5}\,$(ii)$ implies Corollary \ref{c2.10}. 
\end{proof}

\begin{corollary} \lb{c2.11}
Let $A \in \bbC^{m \times m}$ be conditionally positive semidefinite and introduce 
$F \colon \bbR^n \to \bbC^{m \times m}$ by 
\begin{equation}
F(x) = A, \quad x \in \bbR^n.
\end{equation} 
Then $F$ is conditionally positive semidefinite in the sense of 
Definition \ref{d2.4}\,$(ii)$.
\end{corollary}
\begin{proof}
By Lemma \ref{l2.2}, for all $t > 0$, $\exp_H(t A) \geq 0$ is positive semidefinite. Thus, by 
Corollary \ref{c2.10}, for all $t > 0$, $\exp_H(t F)(x) = \exp_H(t A)$, $ x \in \bbR^n$, is positive semidefinite. Hence, by Theorem \ref{t2.7}, $F$ is conditionally positive semidefinite. 
\end{proof}

Corollaries \ref{c2.10} and \ref{c2.11} indeed verify compatibility of Definitions \ref{d2.1} and \ref{d2.4}. 
For other elementary examples of conditionally positive semidefinite matrix-valued functions we 
refer to Example \ref{e4.17}. 

The classical (i.e., scalar-valued situation $m=1$) version of Schoenberg's theorem, at first sight,  suggests an alternative ``weak'' definition of conditionally positive semidefinite functions 
(cf.\ also \cite{Lo67}) as follows: 

\begin{definition} \lb{d2.12}
Let $F \colon \bbR^n \to \bbC^{m \times m}$. Then $F$ is called {\bf weakly conditionally positive 
semidefinite} if for all $N \in \bbN$, $x_p \in \bbR^n$, $1 \leq p \leq N$, and all 
$f = (f_1,\dots,f_m)^\top \in \bbC^m$, the matrix 
$\{(f, F(x_p - x_q) f)_{\bbC^m}\}_{1 \leq p,q \leq N} \in \bbC^{N \times N}$ is conditionally positive semidefinite, that is, for all $c_p \in \bbC$, $1 \leq p \leq N$, with $\sum_{p=1}^N c_p = 0$,
one has 
\begin{equation}
\sum_{p,q=1}^N \sum_{j,k=1}^m \ol{c_p} \, \ol{f_j} \, F(x_p - x_q)_{j,k} f_k c_q \geq 0.  
\end{equation}
\end{definition} 

We will conclude this section by showing via a simple example that Definitions \ref{d2.4}\,$(ii)$ and \ref{d2.12} are inequivalent.

\begin{example} \lb{e2.13}
Consider 
\begin{equation}
A = \begin{pmatrix} \ln(1/2) & 0 \\ 0 & \ln(1/2) \end{pmatrix} 
\end{equation}
and introduce $F \colon \bbR^n \to \bbC^{2 \times 2}$ by
\begin{equation} 
F(x) = A, \quad x \in \bbR^n. 
\end{equation}
Then, for all $N \in \bbN$, $x_p \in \bbR^n$, $c_p \in \bbC$, $1 \leq p \leq N$, with 
$\sum_{p=1}^N c_p = 0$, and all $f = (f_1, f_2)^\top \in \bbC^2$,
\begin{equation}
\sum_{p,q=1}^N \sum_{j,k=1}^2 \ol{c_p} \, \ol{f_j} \, F(x_p - x_q)_{j,k} f_k c_q 
= (f, A f)_{\bbC^2} \sum_{p,q=1}^N \ol{c_p} \, c_q = 0, 
\end{equation}
and hence $F$ is weakly conditionally positive semidefinite. On the other hand, 
\begin{equation}
\exp_H(F) = \exp_H(A) = \begin{pmatrix} 1/2 & 1 \\ 1 & 1/2 \end{pmatrix}, \quad 
x \in \bbR^n. 
\end{equation}
However, $\exp_H(A)$ has a simple negative eigenvalue $\lambda_1 = -1/2$; denoting by 
$v_1 \in \bbC^2$ an associated normalized eigenvector, then for all $N \in \bbN$, $x_p \in \bbR^n$, $c_p \in \bbC$, $1 \leq p \leq N$, one computes 
\begin{align} 
\begin{split}
\sum_{p,q=1}^N \ol{c_p} \, (v_1, \exp_H(F)(x_p - x_q) v_1)_{\bbC^2} c_q 
&= \sum_{p,q=1}^N \ol{c_p} \, (v_1, \exp_H(A)v_1)_{\bbC^2} c_q    \\
&= - \f{1}{2} \bigg|\sum_{p=1}^N c_p\bigg|^2 \leq 0.  
\end{split} 
\end{align}
In particular, as long as $\sum_{p=1}^N c_p \neq 0$, then 
\begin{equation} 
\sum_{p,q=1}^N \ol{c_p} \, (v_1, \exp_H(F)(x_p - x_q) v_1)_{\bbC^2} c_q < 0, 
\end{equation}
and hence $\exp_H(F)$ is {\bf not} positive semidefinite by Lemma \ref{l2.5}\,$(ii)$. Consequently, 
$F$ is {\bf not} conditionally positive semidefinite by Theorem \ref{t2.6} and Definitions \ref{d2.4}\,$(ii)$ 
and \ref{d2.12} are indeed inequivalent.
\end{example}

\begin{remark} \lb{r2.14}
There are other non-equivalent extensions of scalar conditionally positive semidefinite functions to the matrix context in the literature. One of the principal goals in this paper is to extend the classical results Theorem \ref{t1.3} and Theorem \ref{t1.4}\,$(i)$--$(iii)$ to the matrix context. So we chose to use the more restrictive definition of matrix valued conditionally positive semidefinite functions in 
Definition \ref{d2.4}.  For treatments of other non-equivalent extensions of scalar conditionally positive semidefinite functions to the matrix case, see, for instance, \cite[Ch.~II]{GV64}, \cite[Chs.~3, 4]{Xi94}. 
For detailed surveys of the theory of scalar positive semidefinite functions we refer, for example, to \cite{Go14}, \cite{St76}.
\end{remark}

\section{Preliminaries on Operators Associated to Matrix-valued Positive Semidfinite Functions } 
\lb{s3}

In this section we develop the basic material on convolutions involving matrix-valued measures 
and matrix-valued convolution operators needed in our principal Section \ref{s4}. 
We rely on \cite[Sect.~2]{Ch02} and \cite[Sects.~2.1, 3.1]{Ch08} (see also \cite{He15}). For readers who are interested in convolution involving operator-valued measures in the infinite dimensional 
Hilbert space context, we refer to \cite{GJR00}. 

Throughout the remainder of this paper we fix $m \in \bbN$. 

A $\bbC^{m \times m}$-valued measure on $\bbR^n$ is a countably additive function 
$\mu \colon \gB_n \to \bbC^{m \times m}$. Equivalently, $\mu = \{\mu_{j,k}\}_{1 \leq j,k \leq m}$ 
is a $\bbC^{m \times m}$-valued measure on $\bbR^n$, if and only if each entry 
$\mu_{j,k} \colon \gB_n \to \bbC$, $ 1 \leq j,k \leq m$, is a complex measure on $\bbR^n$. The 
variation $|\mu|$ of $\mu$ is defined as the finite nonnegative measure on $\bbR^n$ given by 
\begin{equation}
|\mu|(E) = \sup_{\cP} \bigg\{\sum_{E_{\ell} \in \cP} \|\mu(E_\ell)\|_{\cB(\bbC^m)}\bigg\}, 
\quad E \in \gB_n,
\end{equation}  
where the supremum is taken over all partitions $\cP$ of $E$ into a finite number of pairwise disjoint subsets $E_{\ell} \in \gB_n$. The norm $\| \mu\|$ of $\mu$ is defined by
\begin{equation}
\|\mu\| = |\mu|(\bbR^n),
\end{equation}
and we also introduce the notation
\begin{equation}
N(\mu) = \max_{1 \leq j, k \leq m} \big(|\mu_{j,k}|(\bbR^n)\big) 
= \max_{1 \leq j, k \leq m} \|\mu_{j,k}\|.   \lb{3.3} 
\end{equation}

A function $f = \{f_{j,k}\}_{1 \leq j,k \leq m} \colon \bbR^n \to \bbC^{m \times m}$ is called 
$\mu$-integrable if 
\begin{equation}
\int_{\bbR^n} f(x)_{j,k} \, d\mu_{r,s}(x), \quad 1 \leq j,k,r,s \leq m,
\end{equation}
exist, in which case one defines for all $E \in \gB_n$, the integral 
\begin{align} 
& \int_E f(x) \, d\mu(x) = \bigg\{\bigg(\int_E f(x) \, d\mu(x)\bigg)_{j,k}\bigg\}_{1 \leq j,k \leq m},   \\
& \bigg(\int_E f(x) \, d\mu(x)\bigg)_{j,k} = \sum_{\ell = 1}^m \int_E f(x)_{j,\ell} \, d\mu_{\ell,k}(x), 
\quad 1 \leq j,k \leq m. 
\end{align} 
Then, for all $\mu$-integrable functions $f$, 
\begin{equation}
\bigg\| \int_E f(x) \, d\mu(x)\bigg\|_{\cB(\bbC^m)} \leq \int_E \|f(x)\|_{\cB(\bbC^m)} \, d|\mu|(x), 
\quad E \in \gB_n.    \lb{3.6} 
\end{equation}

Next, we introduce $\cM(\bbR^n, \bbC^{m \times m})$ as the space of all (finite) measures on 
$\bbR^n$ of the form, $\mu \colon \gB_n \to (\bbC^{m \times m}, \| \cdot \|_{\cB(\bbC^m)})$. As 
shown in \cite[Lemma~5]{Ch02}, there exists a linear, isometric order isomorphism between 
$\cM(\bbR^n, \bbC^{m \times m})$ and the dual space of $C_{\infty}(\bbR^n, \bbC^{m \times m})$,  
the duality pairing 
$\langle \, \cdot \, ,\, \cdot \, \rangle \colon C_{\infty}(\bbR^n, \bbC^{m \times m}) \times 
\cM(\bbR^n, \bbC^{m \times m})$ being given by 
\begin{equation}
\langle f, \mu \rangle = \tr_{\bbC^m} \bigg(\int_{\bbR^n} f(x) \, d\mu(x)\bigg) 
= \sum_{j,k = 1}^m \int_{\bbR^n} f(x)_{j,k} \, d\mu_{k,j}(x).     \lb{3.7} 
\end{equation}

Given $\mu \in \cM(\bbR^n, \bbC^{m \times m})$ and a $\mu$-integrable 
$f \colon \bbR^n \to \bbC^{m \times m}$, we define their convolution by  
\begin{equation}
f * \mu \colon \begin{cases} \bbR^n \to \bbC^{m \times m}, \\ 
x \mapsto (f * \mu)(x) = \int_{\bbR^n} f(x-y) \, d\mu(y), 
\end{cases} \quad x \in \bbR^n.    \lb{3.8}
\end{equation}
Moreover, for $p \in [1,\infty)$ we introduce 
\begin{align}
\begin{split} 
L^p(\bbR^n, \bbC^{m \times m}) = \bigg\{ & f\colon \bbR^n \to \bbC^{m \times m} \, \text{measurable} 
\, \bigg| \, \\
& \|f\|_{p,m} = \bigg(\int_{\bbR^n} \| f(x) \|_{\cB(\bbC^m)}^p \, d^n x\bigg)^{1/p} < \infty\bigg\},   
\end{split} 
\end{align}
and similarly, for $p=\infty$,
\begin{align}
\begin{split} 
L^{\infty}(\bbR^n, \bbC^{m \times m}) = \big\{ & f\colon \bbR^n \to \bbC^{m \times m} \, \text{measurable} 
\, \big| \, \\
& \|f\|_{\infty,m} = {\rm ess.sup}_{x \in \bbR^n} \| f(x) \|_{\cB(\bbC^m)} < \infty\big\}.    
\end{split} 
\end{align}

Thus, one estimates,
\begin{align}
\|(f * \mu)(x)\|_{\cB(\bbC^m)} &= \bigg\|\int_{\bbR^m} f(x-y) \, d\mu(y)\bigg\|_{\cB(\bbC^m)}  \no \\
& \leq \int_{\bbR^n} \|f(x-y)\|_{\cB(\bbC^m)} \, d |\mu|(y)   \no \\
& \leq \bigg(\int_{\bbR^n} \|f(x-y)\|_{\cB(\bbC^m)}^p \, d |\mu|(y)\bigg)^{1/p} \big[|\mu|(\bbR^n)\big]^{1/p'}, 
\end{align}
with $p^{-1} + (p')^{-1} = 1$, and hence,
\begin{align}
\|f * \mu\|_{p,m} &= \bigg(\int_{\bbR^n} \|(f * \mu)(x)\|_{\cB(\bbC^m)}^p \, d^nx\bigg)^{1/p}  \no \\
& \leq \bigg(\int_{\bbR^n} \int_{\bbR^n} \|f(x-y)\|_{\cB(\bbC^m)}^p \, d|\mu|(y) \, d^n x\bigg)^{1/p} 
\big[|\mu|(\bbR^n)\big]^{1/p'}   \no \\
& = \bigg(\int_{\bbR^n} \int_{\bbR^n} \|f(x-y)\|_{\cB(\bbC^m)}^p \, d^n x \, d|\mu|(y)\bigg)^{1/p} 
\big[|\mu|(\bbR^n)\big]^{1/p'}   \no \\
& = |\mu|(\bbR^n) \|f\|_{p,m}, \quad p \in [1,\infty).     \lb{3.13}
\end{align}

Thus, for $\mu \in \cM(\bbR^n, \bbC^{m \times m})$ one can introduce the associated convolution operator $T_{\mu} \in \cB\big( L^p(\bbR^n, \bbC^{m \times m})\big)$, $p \in [1,\infty)$,  by
\begin{equation}
T_{\mu} f = f * \mu, \quad f \in L^p(\bbR^n, \bbC^{m \times m}),
\end{equation}
implying (cf.\ \eqref{3.13})
\begin{equation}
\|T_{\mu}\|_{\cB( L^p(\bbR^n, \bbC^{m \times m}))} \leq |\mu|(\bbR^n), \quad p \in [1,\infty).   
\lb{3.14}
\end{equation}

Next, we introduce the following equivalent norm in $L^1(\bbR^n, \bbC^{m \times m})$,
\begin{equation}
\interleave f \interleave_{1,m} := \sum_{j,k=1}^m \|f_{j,k}\|_{1}, \quad 
f \in L^1(\bbR^n, \bbC^{m \times m}),    \lb{3.16}
\end{equation}
such that
\begin{equation}
(c_m')^{-1} \interleave f \interleave_{1,m} \leq \|f\|_{1,m} \leq c_m'  \interleave f \interleave_{1,m}, 
\quad f \in L^1(\bbR^n, \bbC^{m \times m}),
\end{equation} 
where $c_m' \geq 1$ is chosen such that 
\begin{equation}
(c_m')^{-1} \sum_{j,k=1}^m |A_{j,k}| \leq \|A\|_{\cB(\bbC^m)} \leq c_m' \sum_{j,k=1}^m |A_{j,k}|, 
\quad A \in \bbC^{m \times m}.  
\end{equation}
Similarly, introducing the following equivalent norm in $L^2(\bbR^n, \bbC^{m \times m})$,
\begin{equation}
\interleave f \interleave_{2,m} := \sum_{j,k=1}^m \|f_{j,k}\|_{2}, \quad 
f \in L^2(\bbR^n, \bbC^{m \times m}),       \lb{3.19a}
\end{equation}
there exists $c_m'' \geq 1$ such that 
\begin{equation}
(c_m'')^{-1} \interleave f \interleave_{2,m} \leq \|f\|_{2,m} \leq c_m'' \interleave f \interleave_{2,m}, 
\quad f \in L^2(\bbR^n, \bbC^{m \times m}).  
\end{equation} 
In addition, we also introduce the following equivalent norm in $L^{\infty}(\bbR^n, \bbC^{m \times m})$,
\begin{equation}
\interleave f \interleave_{\infty,m} := \max_{1 \leq j,k \leq m} \|f_{j,k}\|_{\infty}, \quad 
f \in L^{\infty}(\bbR^n, \bbC^{m \times m}),       \lb{3.21a}
\end{equation}
then there exists $d_m \geq 1$ such that 
\begin{equation}
(d_m)^{-1} \interleave f \interleave_{\infty,m} \leq \|f\|_{\infty,m} 
\leq d_m \interleave f \interleave_{\infty,m}, \quad f \in L^{\infty}(\bbR^n, \bbC^{m \times m}).  
\end{equation} 
In the special case $m=1$ we will omit the extra subscript $1$ in \eqref{3.16}, \eqref{3.19a}, 
and \eqref{3.21a}.

For future purpose in Section \ref{s4} we now also introduce 
$L^2(\bbR^n, \bbC^{m \times m}_{\rm HS})$, where $\bbC^{m \times m}_{\rm HS}$ denotes the 
space $(\bbC^{m \times m}, \| \, \cdot \, \|_{\cB_2(\bbC^m)})$ (i.e., the operator norm 
$\|\, \cdot \,\|_{\cB(\bbC^m)}$ is now replaced by the Hilbert--Schmidt norm 
$\|\, \cdot \,\|_{\cB_2(\bbC^m)}$), as follows: First, $\bbC^{m \times m}$ can be identified with 
$\bbC^{m^2}$, and then the standard Euclidean norm on $\bbC^{m^2}$ 
becomes the Hilbert--Schmidt norm $\| \, \cdot \, \|_{\cB_2(\bbC^m)}$ (cf., e.g., \cite[p.~93]{Bh97}), 
and hence the space $\cB(\bbC^{m \times m}_{\rm HS})$ can be identified with 
$\bbC^{m^2 \times m^2}$. Summarizing,
\begin{equation}
\bbC^{m \times m}_{\rm HS} \simeq \cB_2(\bbC^m) \simeq \bbC^{m^2}, \quad 
\cB(\bbC^{m \times m}_{\rm HS}) \simeq \cB(\bbC^{m^2}) \simeq \bbC^{m^2 \times m^2}. 
\end{equation}
Thus, we introduce 
\begin{align}
& L^2(\bbR^n, \bbC^{m \times m}_{\rm HS}) = \bigg\{f\colon \bbR^n \to \bbC^{m \times m} \, \text{measurable} \, \bigg| \,     \no \\
& \hspace*{3cm} \|f\|_{L^2(\bbR^n, \, \bbC^{m \times m}_{\rm HS})} 
= \bigg(\int_{\bbR^n} \| f(x) \|_{\cB_2(\bbC^m)}^2 \, d^n x\bigg)^{1/2}    \\
& \hspace*{5.4cm}
= \bigg(\int_{\bbR^n} \sum_{j,k=1}^m |f_{j,k}(x)|^2 \, d^nx\bigg)^{1/2} < \infty\bigg\},    \no 
\end{align}
so that as sets, $L^2(\bbR^n, \bbC^{m \times m})$ and $L^2(\bbR^n, \bbC^{m \times m}_{\rm HS})$ coincide, however, the norms (scalar products) employed differ between them. The classical Plancherel 
theorem then yields
\begin{equation}
\|f\|_{L^2(\bbR^n, \, \bbC^{m \times m}_{\rm HS})} = 
\|f^{\wedge}\|_{L^2(\bbR^n, \, \bbC^{m \times m}_{\rm HS})}, \quad 
f \in L^2(\bbR^n, \bbC^{m \times m}_{\rm HS}).    \lb{3.25} 
\end{equation}

Next, we define left translations $L_x$, $x \in \bbR^n$, acting on 
$f \colon \bbR^n \to \bbC^{m \times m}$, via 
\begin{equation}
(L_x f)(y) = f(y-x), \quad y \in \bbR^n. 
\end{equation}

\begin{definition} \lb{d3.1}
Let $T \in \cB\big( L^p(\bbR^n, \bbC^{m \times m})\big)$, $p \in [1,\infty)\cup \{\infty\}$.~Then 
$T$ is called $\bbC^{m \times m}$--{\bf linear} if 
\begin{equation}
T(Af) = A(Tf), \quad A \in \bbC^{m \times m}, \; f \in  L^p(\bbR^n, \bbC^{m \times m}).
\end{equation}
\end{definition}

\begin{proposition} [{\cite[p.~27]{Ch08}}] \lb{p3.2}
Let $p \in [1,\infty)$ and $\mu \in \cM(\bbR^n, \bbC^{m \times m})$.~Then   
$T_{\mu} \in \cB\big( L^p(\bbR^n, \bbC^{m \times m})\big)$ is $\bbC^{m \times m}$--linear and  \begin{equation}
L_x T_\mu f = T_\mu L_x f, \quad x \in \bbR^n, \; f \in  L^p(\bbR^n, \bbC^{m \times m}). 
\end{equation}
\end{proposition}

\begin{proposition} [{\cite[Proposition~3.1.10, Corollary~3.1.11]{Ch08}}] \lb{p3.3} ${}$ \\ 
$(i)$ Let $p \in [1,\infty)$ and assume that $T \in \cB\big( L^p(\bbR^n, \bbC^{m \times m})\big)$ is 
$\bbC^{m \times m}$--linear. Then the following assertions $(\alpha)$ and $(\beta)$ are equivalent: 
\\[1mm] 
$(\alpha)$ $T = T_{\mu}$ for some $\mu \in \cM(\bbR^n, \bbC^{m \times m})$. \\[1mm] 
$(\beta)$ $L_x T = T L_x$, $x \in \bbR^n$, and 
$T \in \cB\big( C_0(\bbR^n, \bbC^{m \times m}),  C_{b}(\bbR^n, \bbC^{m \times m})\big)$. \\[1mm] 
$(ii)$ Assume that $T \in \cB\big( L^1(\bbR^n, \bbC^{m \times m})\big)$ is  
$\bbC^{m \times m}$--linear. Then the following assertions $(\gamma)$ and $(\delta)$ are equivalent: 
\\[1mm] 
$(\gamma)$ $T = T_{\mu}$ for some $\mu \in \cM(\bbR^n, \bbC^{m \times m})$. \\[1mm] 
$(\delta)$ $L_x T = T L_x$, $x \in \bbR^n$.  
\end{proposition}

Next, given $F \in L^{\infty} (\bbR^n, \bbC^{m \times m})$, we define the associated operator 
$F(-i \nabla) \in \cB\big(L^2(\bbR^n, \bbC^{m \times m})\big)$ by
\begin{equation}
F(-i \nabla) f = \big(f^{\wedge} F\big)^{\vee}, \quad f \in L^2 (\bbR^n, \bbC^{m \times m}). 
\end{equation}
More generally, if $F \in L^1_{\loc}(\bbR^n, \bbC^{m \times m})$, one introduces the maximally 
defined operator of right multiplication by $F$ in $L^2(\bbR^n, \bbC^{m \times m})$, denoted 
by $M_F$, by
\begin{align}
\begin{split} 
& (M_F f)(x) = f(x) F(x),  \\
& \, f \in \dom(M_F) = \big\{g \in L^2(\bbR^n, \bbC^{m \times m}) \, \big | \, 
g F \in L^2(\bbR^n, \bbC^{m \times m})\big\}, 
\end{split} 
\end{align} 
and then defines $F(- i \nabla)$ as the closed operator in $L^2(\bbR^n, \bbC^{m \times m})$ via
\begin{equation} 
F(- i \nabla) f = \cF^{-1} (M_F(\cF f)) 
\end{equation}
(cf.\ \eqref{1.15}, \eqref{1.16} and their unitary extensions to $L^2(\bbR^n, \bbC^{m \times m}_{HS})$ 
as indicated in \eqref{3.25}).

\begin{lemma} \lb{l3.4} 
Suppose $F \in L^{\infty} (\bbR^n, \bbC^{m \times m})$, then
\begin{equation}
L_x F(- i \nabla) = F(- i \nabla) L_x, \quad x \in \bbR^n.   \lb{3.19}
\end{equation}
\end{lemma}

Since $\cS(\bbR^n, \bbC^{m \times m})$ is dense in $L^2(\bbR^n, \bbC^{m \times m})$, and 
all operators in \eqref{3.19} are bounded, it suffices to prove \eqref{3.19} for 
$f \in \cS(\bbR^n, \bbC^{m \times m})$. The latter follows from a straightforward calculation. 

For future purpose we also recall the following results: Introducing 
\begin{equation}
j_a(x) = e^{- a |x|}, \quad a > 0, \; x \in \bbR,  
\end{equation}
one verifies 
\begin{equation}
j_a^{\wedge}(y) = \f{1}{(2 \pi)^{1/2}} \,  \f{2a}{y^2 + a^2}, \quad y \in \bbR. 
\end{equation}
Similarly, introducing 
\begin{equation}
k_a(x) = \prod_{\ell=1}^n j_a(x_{\ell}), \quad x \in \bbR^n,   \lb{3.22} 
\end{equation}
one obtains,
\begin{equation}
k_a^{\wedge}(y) = \f{1}{(2 \pi)^{n/2}} \, \prod_{\ell=1}^n  \f{2a}{y_{\ell}^2 + a^2}, \quad y \in \bbR^n, 
\lb{3.23}
\end{equation}
and hence
\begin{equation}
\big\|k_a^{\wedge}\big\|_{1} = \int_{\bbR^n} \big|k_a^{\wedge} (y)\big| \, d^n y = (2 \pi)^{n/2}.  
\lb{3.24}
\end{equation}

\begin{lemma} \lb{l3.5} 
Let $a > 0$ and introduce the following diagonal matrix 
\begin{align} 
M_a (x) = k_a(x) I_{\bbC^m}, \quad x \in \bbR^n.  
\end{align}  
Then there exists $c_m \geq 1$ such that 
\begin{equation}
\big\|\big(M_a^{\wedge} F\big)^{\vee}\big\|_{\infty,m} \leq c_m^2 \, \|F\|_{\infty,m}, \quad 
F \in L^{\infty} (\bbR^n, \bbC^{m \times m}). 
\end{equation}
\end{lemma}
\begin{proof}
Recalling the definition of $\| \cdot \|_{\max}$ in \eqref{1.9}, there exists $c_m \geq 1$ such 
that
\begin{equation}
c_m^{-1} \|A\|_{\cB(\bbC^m)} \leq \|A\|_{\max} \leq c_m \|A\|_{\cB(\bbC^m)}, \quad 
A \in \bbC^{m \times m}.   \lb{3.27} 
\end{equation}
Next, let $x \in \bbR^n$, $1 \leq j,k \leq m$, then
\begin{align}
\big|\big(M_a^{\wedge} F\big)^{\vee}(x)_{j,k}\big| & = 
(2 \pi)^{-n/2} \bigg|\int_{\bbR^n} e^{i (x \cdot y)} \big(M_a^{\wedge} F\big)(y)_{j,k} \, d^n y\bigg| 
\no \\
& = (2 \pi)^{-n/2} \bigg|\int_{\bbR^n} e^{i (x \cdot y)} k_a^{\wedge}(y) F(y)_{j,k} \, d^n y\bigg| 
\no \\
& \leq (2 \pi)^{-n/2} \Big[{\rm ess.sup}_{y \in \bbR^n} \big(\|F(y)\|_{\max}\big)\Big] 
\int_{\bbR^n} \big|k_a^{\wedge}(y)\big| \, d^n y   \no \\
& = {\rm ess.sup}_{y \in \bbR^n} \big(\|F(y)\|_{\max}\big),
\end{align}
employing \eqref{3.24}. Thus, 
\begin{equation}
\big\|\big(M_a^{\wedge} F\big)^{\vee}(x)\big\|_{\max} \leq 
{\rm ess.sup}_{y \in \bbR^n} \big(\|F(y)\|_{\max}\big), \quad 
x \in \bbR^n, \; F \in L^{\infty} (\bbR^n, \bbC^{m \times m}), 
\end{equation}
and hence 
\begin{align}
\big\|\big(M_a^{\wedge} F\big)^{\vee}\big\|_{\infty,m} &= {\rm ess.sup}_{x \in \bbR^n} 
\big\|\big(M_a^{\wedge} F\big)^{\vee}(x)\big\|_{\cB(\bbC^m)}   \no \\
& \leq c_m \, {\rm ess.sup}_{x \in \bbR^n} 
\big\|\big(M_a^{\wedge} F\big)^{\vee}(x)\big\|_{\max}    \no \\
& \leq c_m \, {\rm ess.sup}_{y \in \bbR^n} \big(\|F(y)\|_{\max}\big)   \no \\ 
& \leq c_m^2 \, {\rm ess.sup}_{y \in \bbR^n} \big(\|F(y)\|_{\cB(\bbC^m)}\big)    \no \\
& = c_m^2 \, \|F\|_{\infty,m}. 
\end{align}
\end{proof}

In the following we use the notation $0 \leq g \in L^2(\bbR^n, \bbC^{m \times m})$ if 
$g \in L^2(\bbR^n, \bbC^{m \times m})$ and $g(x) \geq 0$ (i.e., $g(x) \in \bbC^{m \times m}$ 
is positive semidefinite) for (Lebesgue) a.e.~$x \in \bbR^n$.

\begin{definition} \lb{d3.6}
Let $T \in \cB\big( L^2(\bbR^n, \bbC^{m \times m})\big)$.~Then $T$ is called {\bf positivity preserving} 
$($in $ L^2(\bbR^n, \bbC^{m \times m})$$)$ if for any $0 \leq f \in L^2(\bbR^n, \bbC^{m \times m})$ 
also $T f \geq 0$. 
\end{definition}

As will be shown in Lemma \ref{l3.11}, for $T$ to be positivity preserving it suffices to take 
$0 \leq f \in C_0^{\infty}(\bbR^n, \bbC^{m \times m})$ in Definition \ref{d3.6}. 

\begin{lemma} \lb{l3.7}
Suppose that $F \in L^{\infty} (\bbR^n, \bbC^{m \times m})$ and 
$F(-i \nabla)$ is positivity preserving in $L^2(\bbR^n, \bbC^{m \times m})$.~Then, with $c_m \geq 1$ 
as in \eqref{3.27}, 
\begin{equation}
{\rm ess.sup}_{x \in \bbR^n} \|(F(- i \nabla) f)(x)\|_{\max} \leq 2 c_m^4 \, \|F\|_{\infty,m}    \lb{3.30} 
\end{equation}
for all $f \in L^{\infty}(\bbR^n, \bbC^{m \times m})$ satisfying 
\begin{align}
& (i) \;\;\;\, \supp \, (f) \, \text{ is compact.}    \no \\
& (ii) \;\; \sup_{x \in \bbR^n} \|f(x)\|_{\max} \leq 1. \\
&(iii) \; f(x) \geq 0 \, \text{ for a.e.~$x \in \bbR^n$.}    \no
\end{align}
\end{lemma} 
\begin{proof}
By the spectral theorem one obtains for a.e.~$x \in \bbR^n$,
\begin{equation}
0 \leq f(x) \leq \|f(x)\|_{\cB(\bbC^m)} I_{\bbC^m} \leq c_m \|f(x)\|_{\max} I_{\bbC^m} 
\leq c_m I_{\bbC^m}, 
\end{equation}
employing $c_m \geq 1$ in \eqref{3.27}. Since $\supp\,(f)$ is compact, there exists a sufficiently 
small $a>0$ such that for a.e.~$x \in \bbR^n$, 
\begin{equation}
0 \leq f(x) \leq 2 c_m k_a(x) I_{\bbC^m},     \lb{3.34} 
\end{equation}
with $k_a$ introduced in \eqref{3.22}.\footnote{Actually, the factor $2$ in \eqref{3.34} can be 
replaced by $1 + \varepsilon$ for $0 < \varepsilon$ sufficiently small, provided that we choose  
$0 < a = a(\varepsilon)$ sufficiently small, but since this plays no role in the following, we ignore this improvement.} Since $F(- i \nabla)$ is positivity preserving by hypothesis,
\begin{equation}
0 \leq F(- i \nabla) f \leq 2 c_m F(- i \nabla) (k_a I_{\bbC^m}),
\end{equation}
implying 
\begin{equation}
\|(F(- i \nabla) f)(x)\|_{\cB(\bbC^m)} \leq 2 c_m \|(F(- i \nabla) (k_a I_{\bbC^m}))(x)\|_{\cB(\bbC^m)}
\end{equation}
for a.e.~$x \in \bbR^n$. Thus,
\begin{align}
& {\rm ess.sup}_{x \in \bbR^n} \|(F(- i \nabla)f)(x)\|_{\max}     \no \\ 
& \quad \leq c_m \, 
{\rm ess.sup}_{x \in \bbR^n} \|(F(- i \nabla)f)(x)\|_{\cB(\bbC^m)}   \no \\
& \quad \leq 2 c_m^2 \, {\rm ess.sup}_{x \in \bbR^n} 
\|(F(- i \nabla) (k_a I_{\bbC^m}))(x)\|_{\cB(\bbC^m)}   \no \\
& \quad = 2 c_m^2 \, \|F(- i \nabla) (k_a I_{\bbC^m})\|_{\infty,m}     \no \\
& \quad = 2 c_m^2 \, \big\|\big(M_a^{\wedge} F\big)^{\vee}\big\|_{\infty,m}    \no \\
& \quad \leq 2 c_m^4 \, \|F\|_{\infty,m}, 
\end{align}
applying Lemma \ref{l3.5}. 
\end{proof}

Next, let $A \in \cB(\cH)$ and denote, as usual,
\begin{equation}
\Re(A) = 2^{-1} (A + A^*), \quad \Im(A) = (2i)^{-1} (A - A^*).    \lb{3.37}
\end{equation}
Since $\Re(A)$ and $\Im(A)$ are self-adjoint in $\cH$, we define their positive and negative parts, denoted by $\Re(A)_{\pm}$ and $\Im(A)_{\pm}$, as well as $|\Re(A)|$ and $|\Im(A)|$, with the help of the spectral theorem (with $|T| = (T^* T)^{1/2}$, $T \in \cB(\cH)$), and hence obtain,
\begin{equation}
\Re(A)_{\pm} = 2^{-1} [|\Re(A)| \pm \Re(A)], \quad \Im(A)_{\pm} = 2^{-1} [|\Im(A)| \pm \Im(A)].   
\lb{3.38}  
\end{equation}
Moreover, since $\|T\|_{\cB(\cH)} = \| |T| \|_{\cB(\cH)}$, 
one obtains (with $T = \Re(A)$),
\begin{equation}
\|\Re(A)_{\pm}\|_{\cB(\cH)} \leq \|A\|_{\cB(\cH)}, \quad \|\Im(A)_{\pm}\|_{\cB(\cH)} \leq \|A\|_{\cB(\cH)}.  
\lb{3.39} 
\end{equation}

Next, we drop the the nonnegativity hypothesis $(iii)$ in Lemma \ref{l3.7} and hence obtain the 
following result.

\begin{lemma} \lb{l3.8}
Suppose that $F \in L^{\infty} (\bbR^n, \bbC^{m \times m})$ and 
$F(-i \nabla)$ is positivity preserving in $L^2(\bbR^n, \bbC^{m \times m})$.~Then, with $c_m$ 
as in \eqref{3.27},
\begin{equation}
{\rm ess.sup}_{x \in \bbR^n} \|(F(- i \nabla) f)(x)\|_{\max} \leq 8 c_m^6 \, \|F\|_{\infty,m}    \lb{3.40} 
\end{equation}
for all $f \in L^{\infty}(\bbR^n, \bbC^{m \times m})$ satisfying 
\begin{align} 
\begin{split} 
& (i) \;\, \supp \, (f) \, \text{ is compact.}     \\
& (ii) \; {\rm ess.sup}_{x \in \bbR^n} \|f(x)\|_{\max} \leq 1. 
\end{split}
\end{align}
\end{lemma} 
\begin{proof}
With $c_m$ as in \eqref{3.27}, one concludes from the latter and from \eqref{3.39} that 
for a.e.~$x \in \bbR^n$, 
\begin{equation}
\|\Re(f(x))_{\pm}\|_{\max} \leq c_m \|\Re(f(x))_{\pm}\|_{\cB(\bbC^m)} \leq c_m \|f(x)\|_{\cB(\bbC^m)} 
\leq c_m^2 \|f(x)\|_{\max}.
\end{equation} 
Thus,
$\Re(f)_{\pm} \colon \bbR^n \to \bbC^{m \times m}$ satisfies \\
$(\alpha)$ $\supp \, (\Re(f)_{\pm})$ is compact. \\[1mm] 
$(\beta)$ ${\rm ess.sup}_{x \in \bbR^n} \|\Re(f(x))_{\pm}\|_{\max} \leq c_m^2$.  \\[1mm] 
$(\gamma)$ $\Re(f(x))_{\pm} \geq 0$ for a.e.~$x \in \bbR^n$.     \\[1mm] 
By Lemma \ref{l3.7}, 
\begin{equation}
{\rm ess.sup}_{x \in \bbR^n} \|(F(- i \nabla) \Re(f)_{\pm})(x)\|_{\max} 
\leq 2 c_m^6 \, \|F\|_{\infty,m},    \lb{3.43} 
\end{equation}
and similarly, 
\begin{equation}
{\rm ess.sup}_{x \in \bbR^n} \|(F(- i \nabla) \Im(f)_{\pm})(x)\|_{\max} 
\leq 2 c_m^6 \, \|F\|_{\infty,m},    \lb{3.44} 
\end{equation}
implying 
\begin{equation}
{\rm ess.sup}_{x \in \bbR^n} \|(F(- i \nabla) f)(x)\|_{\max} 
\leq 8 c_m^6 \, \|F\|_{\infty,m}.    \lb{3.45} 
\end{equation} 
\end{proof}

In order to prove a consequence of Lemma \ref{l3.8}, we need the following auxiliary result.

\begin{lemma} [cf., e.g., {\cite[Theorem~2.29 and p.\ 250]{AF03}}] \lb{l3.9} ${}$ \\[1mm] 
$(i)$ If $f \in L^1(\bbR^n)$, then $f^{\wedge} \in C_{\infty}(\bbR^n)$ and 
$\big\|f^{\wedge}\big\|_{\infty} \leq (2 \pi)^{-n/2} \|f\|_1$.  \\[1mm] 
$(ii)$ Suppose $f \in C_0(\bbR^n)$ with $\supp \, (f) \subseteq \ol{B_n(0,r)}$ for some $r > 0$.~Then 
there exists a sequence of functions $\{f_j\}_{j \in \bbN} \subset C_0^{\infty}(\bbR^n)$, satisfying 
$\supp \, (f_j) \subseteq \ol{B_n(0, 2r)}$, $j \in \bbN$, and $\lim_{j \to \infty}\|f_j - f\|_{\infty} = 0$.
\end{lemma}

\begin{remark}
Let $\sigma \colon \gB_n \to [0, \infty)$ be a finite nonnegative measure on $\bbR^n$ and let 
$\mu \colon \gB_n \to \bbC^{m \times m}$ be the nonnegative matrix-valued measure defined by 
\begin{equation}
\mu(E) = \sigma(E) I_m, \quad E \in \gB_n.
\end{equation}
Then $T_{\mu} \in \cB\big(L^2(\bbR^n, \bbC^{m \times m})\big)$ is positivity preserving. Indeed, 
suppose that $f \in L^2(\bbR^n, \bbC^{m \times m})$, then 
\begin{equation}
\int_{\bbR^n} f(y) \, d\mu(y) = \bigg\{\int_{\bbR^n} f_{j,k}(y) \, d \sigma(y)\bigg\}_{1 \leq j,k \leq m}.
\end{equation}
Hence, if $0 \leq f \in L^2(\bbR^n, \bbC^{m \times m})$, then for all 
$v = (v_1, \dots , v_m)^\top \in \bbC^m$ one obtains 
\begin{equation}
(v, (T_{\mu} f)(x) v)_{\bbC^m} = \int_{\bbR^n} \sum_{j,k = 1}^{m} \ol{v_j} f_{j,k}(x - y) v_k 
\, d \sigma(y) \geq 0.
\end{equation} 
\hfill$\diamond$ 
\end{remark}

\begin{lemma} \lb{l3.12}
Assume that $ 0 \leq f \in C_{\infty}(\bbR^n, \bbC^{m \times m})$.~Then 
there exists a sequence $\{f_j\}_{j \in \bbN} \subset C_0^{\infty}(\bbR^n, \bbC^{m \times m})$ such that $f_j(x) \geq 0$, $j \in \bbN$, and $\lim_{j \to \infty} f_j = f$ in the space 
$(C_{\infty}(\bbR^n, \bbC^{m \times m}), \|\, \cdot \, \|_{\infty,m})$.
\end{lemma}
\begin{proof} 
Clearly one can find a sequence $\{g_j\}_{j \in \bbN} \subset C_0(\bbR^n, \bbC^{m \times m})$ such 
that 
\begin{equation} 
g_j \geq 0, \; j \in \bbN, \, \text{ and } \, \lim_{j \to \infty} g_j = f \, \text{ in } \,   
(C_{\infty}(\bbR^n, \bbC^{m \times m}), \|\, \cdot \, \|_{\infty,m}).    \lb{3.57}
\end{equation}
Indeed, let 
\begin{equation}
k_n \in C_0(\bbR^n), \quad 0 \leq k_n \leq 1, \quad 
k_n(x) = \begin{cases} 1, & 0 \leq |x| \leq n, \\
0, & |x| \geq n+1, \end{cases}   
\end{equation} 
$k_n$ decreasing from $1$ to $0$ as $|x|$ increases from $n$ to $n+1$, 
and put $g_n = k_n f$, 
$n \in \bbN$. Then $g_n \geq 0$ on $\bbR^n$ and $f(x) - g_n(x) =0$ for 
$0 \leq |x| \leq n$. Since 
$\|g_n(x)\|_{\max} \leq \|f(x)\|_{\max}$ and $\lim_{|x| \to \infty} \|f(x)\|_{\max} = 0$, 
one obtains 
\eqref{3.57}. Thus, without loss of generality we may assume that 
$f \in C_0(\bbR^n, \bbC^{m \times m})$. 

Next, we recall the definition of standard Friedrichs 
mollifiers $\{\phi_{\varepsilon}\}_{\varepsilon >0}$ 
(cf., e.g., \cite[p.~36, 37]{AF03}) and introduce 
\begin{equation}
\Phi_{\varepsilon}(x) = \phi_{\varepsilon}(x) I_m, \quad x \in \bbR^n, \; \varepsilon >0. 
\end{equation}
In addition, we define the measure $\sigma_{\varepsilon} \in \cM(\bbR^n, \bbC^{m \times m})$ by 
\begin{equation}
\sigma_{\varepsilon} (E) = \bigg(\int_E \phi_{\varepsilon} (x) \, d^nx\bigg) I_m, \quad 
E \in \gB_n.
\end{equation}
Then, using the fact that $T_{\sigma_{\varepsilon}}$ is positivity 
preserving in $L^2(\bbR^n, \bbC^{m \times m})$, one introduces $f_j = T_{\sigma_{1/j}}f$, 
$j \in \bbN$, and concludes $f_j \geq 0$, $j \in \bbN$. Moreover, 
\begin{equation}
f_j(x)_{k,\ell} = (f_{k,\ell} * \phi_{\varepsilon})(x), \quad x \in \bbR^n, \; j \in \bbN, \; 
1 \leq k, \ell \leq m.
\end{equation} 
By standard properties of mollifiers, $(f_j)_{k,\ell} \in C_0^{\infty}(\bbR^n)$ and 
\begin{equation}
\lim_{j \to \infty} (f_j)_{k,\ell} =f_{k,\ell} \, \text{ in } \, 
(C_0(\bbR^n), \|\, \cdot \, \|_{\infty}), \quad 1 \leq k, \ell \leq m.
\end{equation}
Thus, $0 \leq f_j \in C_0^{\infty}(\bbR^n, \bbC^{m \times m})$ and  $\lim_{j \to \infty} f_j = f$ in  
$(C_{\infty}(\bbR^n, \bbC^{m \times m}), \|\, \cdot \, \|_{\infty,m})$. 
\end{proof}

\begin{corollary} \lb{c3.10}
Suppose that $F \in L^{\infty} (\bbR^n, \bbC^{m \times m})$ and 
$F(-i \nabla)$ is positivity preserving in $L^2(\bbR^n, \bbC^{m \times m})$.~Then 
\begin{equation} 
F(-i \nabla) \colon (C_0(\bbR^n, \bbC^{m \times m}), \| \, \cdot \,\|_{\infty,m}) \to 
(C_b(\bbR^n, \bbC^{m \times m}), \| \, \cdot \,\|_{\infty,m}) \, \text{ continuously.} 
\end{equation} 
In addition, there exists a nonnegative measure $\mu \in \cM(\bbR^n, \bbC^{m \times m})$ such 
that $F(- i \nabla) = T_{\mu}$. 
\end{corollary}
\begin{proof}
Suppose $f \in C_0(\bbR^n, \bbC^{m \times m})$ and $\supp(f) \subseteq \ol{B_n(0,r)}$. Applying  
Lemma \ref{l3.9}\,$(ii)$, there exists a sequence of functions 
$\{f_j\}_{j \in \bbN} \in C_0^{\infty}(\bbR^n, \bbC^{m \times m})$, such that 
$\supp \, (f_j) \subseteq \ol{B_n(0, 2r)}$, $j \in \bbN$, and 
$\lim_{j \to \infty}\|(f_j)_{k,\ell} - f_{k,\ell}\|_{\infty} = 0$, $1 \leq k, \ell \leq m$. 

Without loss of generality we may assume that for each $j \in \bbN$, $(f_j - f)$ satisfies the hypotheses of Lemma \ref{l3.8}. Thus, since 
\begin{equation}
\lim_{j \to \infty} {\rm ess.sup}_{x \in \bbR^n} \|f_j(x) - f(x)\|_{\max} = 0,
\end{equation} 
Lemma \ref{l3.8} yields 
\begin{equation}
\lim_{j \to \infty} {\rm ess.sup}_{x \in \bbR^n} \|(F(- i \nabla) f_j)(x) - (F(- i \nabla) f)(x)\|_{\max} = 0. 
\end{equation} 
Since $f_j \in C_0^{\infty}(\bbR^n, \bbC^{m \times m})$, 
$f_j^{\wedge} \in \cS(\bbR^n, \bbC^{m \times m})$, and hence 
$f_j^{\wedge} F \in L^1(\bbR^n, \bbC^{m \times m})$. Thus, applying Lemma \ref{l3.9}\,$(i)$ 
implies $F(- i \nabla) f_j = \big(f_j^{\wedge} F\big)^{\vee} \in C_{\infty}(\bbR^n, \bbC^{m \times m})$.  
Hence, $F(- i \nabla) f$ is the uniform limit of a bounded sequence 
$\{F(- i \nabla) f_j\}_{j \in \bbN} \subset C_{\infty}(\bbR^n, \bbC^{m \times m})$ and thus  
$F(- i \nabla) f \in C_b(\bbR^n, \bbC^{m \times m})$. Lemma \ref{l3.8} implies that 
$F(- i \nabla)$ maps $(C_0(\bbR^n, \bbC^{m \times m}), \| \, \cdot \, \|_{\infty,m})$ to 
$(C_b(\bbR^n, \bbC^{m \times m}), \| \, \cdot \, \|_{\infty,m})$ continuously. That there exists a 
$\mu \in \cM(\bbR^n, \bbC^{m \times m})$ such that $F(- i \nabla) = T_{\mu}$ follows from 
Proposition \ref{p3.3}\,$(i)$ (upon choosing $T=F(-i\nabla)$ in Proposition \ref{p3.3}\,$(i)$,\,$(\beta)$) and Lemma \ref{l3.4}.~Identifying $\cM(\bbR^n, \bbC^{m \times m})$ with 
$C_{\infty}(\bbR^n, \bbC^{m \times m})^*$, it remains to show that 
\begin{equation}
\tr_{\bbC^m} \bigg(\int_{\bbR^n} f(x) \, d\mu(x)\bigg) \geq 0, \quad 
0 \leq f \in C_{\infty}(\bbR^n, \bbC^{m \times m}). 
\end{equation}
By Lemma \ref{l3.12} it suffices to show that this inequality hods for all 
$0 \leq f \in C_0^{\infty}(\bbR^n, \bbC^{m \times m})$. Thus, let 
$0 \leq f \in C_0^{\infty}(\bbR^n, \bbC^{m \times m})$, then 
$f^{\wedge} \in \cS(\bbR^n, \bbC^{m \times m})$ and hence by Lemma \ref{l3.9}\,$(i)$, 
\begin{equation}
F(- i \nabla) f = \big(f^{\wedge} F\big)^{\vee} \in C_{\infty}(\bbR^n, \bbC^{m \times m}). 
\end{equation}
In addition, since $F(- i \nabla)$ is positivity preserving,
\begin{equation}
0 \leq (F(- i \nabla) f)(0) = (T_{\mu} f)(0) = \int_{\bbR^n} f(- y) \, d\mu(y).
\end{equation}
Thus, 
\begin{equation}
\tr_{\bbC^m} \bigg(\int_{\bbR^n} f(-y) \, d\mu(y)\bigg) \geq 0, 
\end{equation}
and hence $\mu$ is nonnegative. 
\end{proof}

We also add the following auxiliary result.

\begin{lemma} \lb{l3.11}
Let $f \in L^2(\bbR^n, \bbC^{m \times m})$ and suppose $f(x) \geq 0$ for a.e.\ $x \in \bbR^n$. Then 
there exists a sequence $\{f_j\}_{j \in \bbN} \subset C_0^{\infty}(\bbR^n, \bbC^{m \times m})$ such that 
for all $j \in \bbN$, $f_j(x) \geq 0$ for a.e.\ $x \in \bbR^n$, and $\lim_{j \to \infty} \|f_j - f\|_{2,m} = 0$.
\end{lemma}
\begin{proof} 
Let $\phi_{\varepsilon}$, $\Phi_{\varepsilon}$, and $\sigma_{\varepsilon}$, $\varepsilon > 0$, be 
as introduced in  the proof of Lemma \ref{l3.12}, and recall that $T_{\sigma_{\varepsilon}}$ is 
positivity preserving in $L^2(\bbR^n, \bbC^{m \times m})$. Next, let 
$0 \leq f \in L^2(\bbR^n, \bbC^{m \times m})$ and introduce 
\begin{equation}
g_j = (\chi_{[-j,j]^n} I_m) f, \quad j \in \bbN,
\end{equation}
where $\chi_{[-j,j]^n}$ denotes the characteristic function of $[-j, j]^n \subset \bbR^n$. Clearly, 
$0 \leq g_j \in L^2(\bbR^n, \bbC^{m \times m})$, $\supp \, (g_j)$ is compact, $j \in \bbN$, and 
$\lim_{j \to \infty} \|g_j - f\|_{2,m} = 0$. Hence, it suffices to show that if 
$0 \leq g \in L^2(\bbR^n, \bbC^{m \times m})$ and $\supp \, (g)$ is compact, then there exists a 
sequence $\{h_j\}_{j \in \bbN} \subset C_0^{\infty}(\bbR^n, \bbC^{m \times m})$ such that 
$0 \leq h_j$, $j \in \bbN$, and $\lim_{j \to \infty} \|h_j - g\|_{2,m} = 0$. Thus, let 
\begin{equation}
h_j = T_{\sigma_{1/n}}g, \quad j \in \bbN.
\end{equation} 
Then $h_j \geq 0$ since $T_{\sigma_{1/n}}$ is positivity preserving and 
\begin{equation}
h_j(x)_{k,\ell} = (g_{k,\ell} * \phi_{1/n})(x), \quad x \in \bbR^n, \; 1 \leq k, \ell \leq m. 
\end{equation} 
By standard properties of Friedrichs mollifiers (cf., e.g., \cite[p.~36, 37]{AF03}), 
$(h_j)_{k,\ell} \in C_0^{\infty}(\bbR^n)$ and 
\begin{equation}
\lim_{j \to \infty} \|(h_j)_{k,\ell} - g_{k,\ell}\|_2 = 0, \quad 1 \leq k, \ell \leq m,
\end{equation}
implying $\{h_j\}_{j \in \bbN} \subset C_0^{\infty}(\bbR^n, \bbC^{m \times m})$ and 
$\lim_{j \to \infty} \|h_j - g\|_{2,m} = 0$. 
\end{proof}

Introducing the {\it Hadamard product} $A \circ_H B$ of two matrices $A, B \in \bbC^{m \times m}$,  
by
\begin{equation}
(A \circ_H B)_{j,k} = A_{j,k} B_{j,k}, \quad 1 \leq j,k \leq m, 
\end{equation}
we conclude this section with the following remark, addressing the lack of the semigroup property 
of $\exp_H(t F)(- i \nabla)$.

\begin{remark} \lb{r3.13}
Suppose that $F \colon \bbR^n \to \bbC^{m \times m}$ is conditionally positive semidefinite 
satisfying for some $c \in \bbR$, 
\begin{equation}
\Re(F(x)_{j,k}) \leq c \, \text{ for a.e.\ $x \in \bbR^n$, $1 \leq j,k \leq m$.}
\end{equation}
In addition, introduce
\begin{equation}
f(t) = (\exp_H(t F)(- i \nabla)) f, \quad f \in L^2(\bbR^n, \bbC^{m \times m}), \; t \geq 0.
\end{equation}
Then,
\begin{equation}
\f{d}{dt} (f(t)) = \big(f^{\wedge} ((\exp_H(tF)) \circ_H F)\big)^{\vee}, \quad t > 0. 
\end{equation}
\hfill $\diamond$ 
\end{remark}

\section{Operators Associated With Matrix-Valued Positive Semidefinite Functions} 
\lb{s4}

In this section we prove our principal results. In particular, we will prove analogs of the 
classical Theorems \ref{t1.3} and parts $(i)$--$(iii)$ of Theorem \ref{t1.4} in the matrix-valued 
context to the extent possible and along the way introduce the necessary modifications needed to 
obtain such extensions. We also recall Fourier multiplier theorems in the $L^1$ and $L^2$ context 
extending classical results in the scalar case $m=1$ to the matrix-valued situation $m \in \bbN$, 
$m \geq 2$. 
 
We start with the following fact.

\begin{theorem} \lb{t4.1}
Suppose that $F \in C(\bbR^n, \bbC^{m \times m}) \cap L^{\infty} (\bbR^n, \bbC^{m \times m})$ 
and $F(- i \nabla)$ is positivity preserving 
 in $L^2(\bbR^n, \bbC^{m \times m})$.~Then there exists a nonnegative measure 
 $\mu \in \cM(\bbR^n, \bbC^{m \times m})$ such that
 \begin{equation}
 F(x) = \mu^{\wedge}(x), \quad x \in \bbR^n,    \lb{4.1} 
 \end{equation} 
 equivalently,
 \begin{equation}
  F(x) = (2 \pi)^{-n/2} \int_{\bbR^n} e^{- i (x \cdot \xi)} \, d\mu(\xi), \quad x \in \bbR^n, 
 \end{equation}
 holds.
\end{theorem}
\begin{proof}
Define $\phi_{\varepsilon}$ and $\Phi_{\varepsilon}$ as in the proof of Lemma \ref{l3.11} 
and introduce  
\begin{equation}
\Phi_{\varepsilon,x} (y) = \Phi_{\varepsilon} (x-y), \quad x,y \in \bbR^n, \; \varepsilon >0. 
\end{equation}
Suppose $f \in \cS(\bbR^n, \bbC^{m \times m})$, then
\begin{align}
(F(- i \nabla) f)(x) &= \big(f^{\wedge} F\big)^{\vee}(x)   \no \\
&= (2 \pi)^{-n} \int_{\bbR^n} \int_{\bbR^n} e^{i (\xi \cdot(x - \eta))} f(\eta) F(\xi) 
\, d^n \eta \, d^n \xi    \no \\
&= (2 \pi)^{-n} \int_{\bbR^n} \int_{\bbR^n} e^{i (\xi \cdot \omega)} f(x-\omega) F(\xi) 
\, d^n \omega \, d^n \xi   \no \\
&= (2 \pi)^{- n/2} \int_{\bbR^n} (f(x - \cdot))^{\vee}(\xi) F(\xi) \, d^n \xi. 
\end{align}
Introducing $f_{\varepsilon,x} \in \cS(\bbR^n, \bbC^{m \times m})$ by
\begin{equation}
f_{\varepsilon,x} (y) = (\Phi_{\varepsilon,x})^{\wedge} (x-y), \quad x,y \in \bbR^n, \; \varepsilon >0,
\end{equation}
one obtains for $\varepsilon > 0$, 
\begin{align}
(F(- i \nabla) f_{\varepsilon,x})(x) &= (2 \pi)^{- n/2} \int_{\bbR^n} 
(f_{\varepsilon,x}(x - \cdot))^{\vee}(\xi) F(\xi) \, d^n \xi      \no \\
&= (2 \pi)^{- n/2} \int_{\bbR^n} 
\big(\Phi_{\varepsilon,x}^{\wedge}\big)^{\vee}(\xi) F(\xi) \, d^n \xi      \no \\
&= (2 \pi)^{- n/2} \int_{\bbR^n} 
\Phi_{\varepsilon,x}(\xi) F(\xi) \, d^n \xi      \no \\
& \underset{\varepsilon \downarrow 0}{\longrightarrow} (2 \pi)^{- n/2} F(x), \quad x \in \bbR^n.
\end{align}

By Corollary \ref{c3.10}, there exists a nonnegative measure 
 $\mu_0 \in \cM(\bbR^n, \bbC^{m \times m})$ such that $F(- i \nabla) = T_{\mu_0}$. Hence, 
 \begin{align}
& (F(- i \nabla) f_{\varepsilon,x})(x) =  (T_{\mu_0} f_{\varepsilon,x})(x)   
 = (f_{\varepsilon,x} * \mu_0)(x)     \no \\ 
& \quad = \int_{\bbR^n} f_{\varepsilon,x}(x-\eta) \, d \mu_0(\eta)     \no \\
& \quad = \int_{\bbR^n} \Phi_{\varepsilon,x}^{\wedge}(\eta) \, d\mu_0(\eta)    \no \\
& \quad = (2 \pi)^{- n/2} \int_{\bbR^n} \int_{\bbR^n} e^{- i (\eta \cdot \xi)} \Phi_{\varepsilon,x}(\xi) 
\, d^n \xi \, d\mu_0(\eta), \quad x \in \bbR^n, \; \varepsilon > 0.    \lb{4.7}
 \end{align}
Since $\Phi_{\varepsilon,x}$ has compact support and $\mu_{k,\ell}$, $1 \leq k, \ell \leq m$ are 
finite complex measures on $\bbR^n$, one can freely interchange the order of integration in the 
last double integral in \eqref{4.7} to arrive at
\begin{align}
(F(- i \nabla) f_{\varepsilon,x})(x) &= 
(2 \pi)^{- n/2} \int_{\bbR^n} \Phi_{\varepsilon,x}(\xi)  \bigg(\int_{\bbR^n} e^{- i (\xi \cdot \eta)} 
\, d\mu_0(\eta)\bigg) d^n \xi    \no \\
&= \int_{\bbR^n} \Phi_{\varepsilon}(x-\xi) \mu_0^{\wedge}(\xi) \, d^n \xi    \no \\
& \underset{\varepsilon \downarrow 0}{\longrightarrow} \mu_0^{\wedge}(x),  
\quad x \in \bbR.
\end{align}
Thus,  \eqref{4.1} follows with $\mu = (2 \pi)^{n/2} \mu_0$.
\end{proof}

\begin{remark} \lb{r4.1a}
In Appendix \ref{sA} we will prove that that the converse to Theorem \ref{t4.1}, that is, if 
$F = \mu^{\wedge}$ for some nonnegative $\mu \in \cM(\bbR^n, \bbC^{m \times m})$ then 
$F(- i \nabla)$ is positivity preserving in $L^2(\bbR^n, \bbC^{m \times m})$, does {\bf not} hold 
(unless, of course, $\mu$ is of the type $\mu_{\sigma} = \sigma I_{\bbC^m}$ with 
$\sigma \colon \gB_n \to [0,\infty)$ a finite measure). \hfill $\diamond$
\end{remark}

Next, we recall the finite-dimensional special case of an infinite-dimensional 
version of Bochner's theorem (cf.\ Theorem \ref{t1.2}) in connection with locally compact Abelian 
groups due to Berberian \cite{Be66} (see also \cite{FH72}, \cite{FK91}, \cite{Ml83}, \cite{vW68}):

\begin{theorem} [{\cite[p~178, Theorem~3 and Corollary on p.~177]{Be66}}] \lb{t4.2} ${}$ \\[1mm]
Assume that $F \in C(\bbR^n, \bbC^{m \times m}) \cap L^{\infty} (\bbR^n, \bbC^{m \times m})$. 
Then the following conditions $(i)$ and $(ii)$ are equivalent: \\[1mm]
$(i)$ $F$ is positive semidefinite. \\[1mm]
$(ii)$ There exists a nonnegative measure $\mu \in \cM(\bbR^n, \bbC^{m \times m})$ such that 
\begin{equation}
 F(x) = \mu^{\wedge}(x), \quad x \in \bbR^n.    \lb{4.9} 
\end{equation} 
In addition, if one of conditions $(i)$ or $(ii)$ holds, then
\begin{equation}
F(-x) = F(x)^*, \quad \|F(x)\|_{\cB(\bbC^m)} \leq \|F(0)\|_{\cB(\bbC^m)}, \quad x \in \bbR^n. 
\lb{4.10} 
\end{equation} 
\end{theorem}

We note that Berberian \cite[p~178, Theorem~3]{Be66} discusses a seemingly more general result in which boundedness of $F$ is not assumed, it is, however, a consequence of his results. 

Next, we extend the classical $L^1$-multiplier theorem due to Bochner (cf., e.g., 
\cite[Theorem~2.5.8 and p.~143, 144]{Gr08}, \cite[p.~28]{St86}, \cite[p.~29, 30]{SW90}) to the matrix-valued context. An infinite-dimensional version of this result appeared in Gaudry, Jefferies, and 
Ricker \cite[Proposition~3.15 and Corollary~3.20]{GJR00}. For completeness, we present an elementary proof in the matrix-valued case and add the estimates \eqref{4.11} which appear to be new in this context. 

We recall definition \eqref{3.3} of $N(\mu)$ and the definition of 
$\interleave \cdot \interleave_{1,m}$ in \eqref{3.16}.  

\begin{theorem} \lb{t4.3} 
Assume that $F \in L^{\infty}(\bbR^n, \bbC^{m \times m})$.~Then the following conditions $(i)$ and 
$(ii)$ are equivalent: \\[1mm]
$(i)$ $F(- i \nabla)|_{C_0^{\infty}(\bbR^n, \bbC^{m \times m})}$ can be extended to a bounded operator 
$($denoted by the same symbol, for simplicity\,$)$ 
$F(- i \nabla) \in \cB\big(L^1(\bbR^n, \bbC^{m \times m})\big)$. \\[1mm]
$(ii)$ There exists a measure $\mu \in \cM(\bbR^n, \bbC^{m \times m})$ such that 
\begin{equation}
 F(x) = \mu^{\wedge}(x), \quad x \in \bbR^n.     
\end{equation} 
In addition, if one of conditions $(i)$ or $(ii)$ holds, then
\begin{equation}
(2 \pi)^{-n/2} N(\mu) \leq \|F(- i \nabla)\|_{\cB((L^1(\bbR^n, \bbC^{m \times m}), \, 
\interleave \, \cdot \, \interleave_{1,m}))} 
\leq m (2 \pi)^{-n/2} N(\mu).    \lb{4.11}
\end{equation} 
Both estimates in \eqref{4.11} are sharp.
\end{theorem} 
\begin{proof}
First, suppose that condition $(ii)$ holds. Let $f \in C_0^{\infty}(\bbR^n, \bbC^{m \times m})$, then 
\begin{equation}
(F(- i \nabla)f)(x) = (2 \pi)^{-n} \int_{\bbR^n} \int_{\bbR^n} e^{i (\xi \cdot (x - \eta))} f^{\wedge}(\xi) 
\, d \mu(\eta) \, d^n \xi, \quad x \in \bbR^n.   \lb{4.12} 
\end{equation}
Since $f^{\wedge} \in \cS(\bbR^n, \bbC^{m \times m}) \subset L^1(\bbR^n, \bbC^{m \times m})$, 
one can interchange the order of integration in \eqref{4.12} and hence obtains 
\begin{align}
(F(- i \nabla)f)(x) &= (2 \pi)^{-n} \int_{\bbR^n} \int_{\bbR^n} e^{i (\xi \cdot (x - \eta))} f^{\wedge}(\xi) 
\, d^n \xi \, d \mu(\eta)    \no \\
&= (2 \pi)^{-n/2} \int_{\bbR^n} \big(f^{\wedge}\big)^{\vee} (x - \eta) \, d \mu(\eta)   \no \\
&= (2 \pi)^{-n/2} (T_{\mu} f)(x), \quad x \in \bbR^n. 
\end{align}
Thus, applying \eqref{3.8}--\eqref{3.14}, 
\begin{equation}
\|F(- i \nabla)\|_{\cB(L^1(\bbR^n, \bbC^{m \times m}))} \leq (2 \pi)^{-n/2} \|\mu\|,
\end{equation}
implying condition $(i)$. 

To prove the converse implication, we now suppose that condition $(i)$ holds. We introduce, 
$I(j,k) \in \bbC^{m \times m}$ by
\begin{equation}
I(j,k)_{p,q} = \begin{cases}  1 & \text{if $ p=j$ and $q=k$,} \\
0 & \text{if $ p \neq j$ or $q \neq k$,}\end{cases}   \quad 1 \leq j,k,p,q \leq m.   \lb{4.15}
\end{equation}
In addition, let 
\begin{equation}
U(j,k) \colon \begin{cases} L^1(\bbR^n) \to L^1(\bbR^n, \bbC^{m \times m}),  \\
g \mapsto U(j,k) g = g I(j,k), \end{cases} \quad 1 \leq j,k \leq m,    \lb{4.16}
\end{equation}
and 
\begin{equation}
D(j,k) \colon \begin{cases} L^1(\bbR^n, \bbC^{m \times m}) \to L^1(\bbR^n),  \\
f \mapsto D(j,k) f = f_{j,k}, \end{cases} \quad 1 \leq j,k \leq m.    \lb{4.17}
\end{equation}
One verifies that $U(j,k)$ and $D(j,k)$ are bounded for each $1 \leq j,k \leq m$, and hence 
also 
\begin{equation}
P(p,q,j,k) = D(p,q) F(- i \nabla) U(j,k) \colon L^1(\bbR^n) \to L^1(\bbR^n), \quad 
1 \leq j,k,p,q \leq m,     \lb{4.18} 
\end{equation}  
are bounded. Employing the fact that 
\begin{equation}
P(1,k,1,j) g = \big(g^{\wedge} F_{j,k}\big)^{\vee}, \quad g \in L^1(\bbR^n),    \lb{4.19} 
\end{equation}
one infers that the linear operator 
$L^1(\bbR^n) \ni g \mapsto \big(g^{\wedge} F_{j,k}\big)^{\vee} \in L^1(\bbR^n)$ 
is bounded, that is, $F_{j,k}$ is an $L^1(\bbR^n)$-multiplier. By the classical Bochner theorem, 
there exists a (finite) complex measure $\mu_{k,j}$ on $\bbR^n$ such that 
$F_{j,k} = \mu_{j,k}^{\wedge}$. Introducing 
$\mu = \{\mu_{j,k}\}_{1 \leq j,k \leq m} \in \cM(\bbR^n, \bbC^{m \times m})$, then 
$F = \mu^{\wedge}$ and hence condition $(ii)$ holds. 

Next we turn to the lower bound in \eqref{4.11}. Choose $p,q \in \{1,\dots,m\}$ such that 
\begin{equation}
N(\mu) = |\mu_{p,q}|(\bbR^n). 
\end{equation}
Since $F_{p,q} = \mu_{p,q}^{\wedge}$, the classical (i.e., scalar-valued) $L^1$-multiplier theorem 
applies and hence yields that $F_{p,q}(- i \nabla)|_{C_0^{\infty}(\bbR^n)}$ can be extended to a bounded operator 
$F_{p,q}(- i \nabla) \in \cB\big(L^1(\bbR^n)\big)$ with norm
\begin{equation}
\|F_{p,q}(- i \nabla)\|_{\cB(L^1(\bbR^n))} = (2 \pi)^{-n/2} \|\mu_{p,q}\| 
= (2 \pi)^{-n/2} |\mu_{p,q}|(\bbR^n).    
\end{equation}
Thus, there exists a sequence $\{f_{\ell}\}_{\ell \in \bbN}$ in $L^1(\bbR^n)$ with $\|f_{\ell}\|_1 = 1$, 
$\ell \in \bbN$, such that 
\begin{equation}
\lim_{\ell \to \infty} \|F_{p,q}(- i \nabla) f_{\ell}\|_1 = (2 \pi)^{-n/2} |\mu_{p,q}|(\bbR^n). 
\end{equation}
Since $C_0^{\infty}(\bbR^n)$ is dense in $L^1(\bbR^n)$, we can assume that 
$f_{\ell} \in C_0^{\infty}(\bbR^n)$, $\ell \in \bbN$. Introduce (cf.\ \eqref{4.16})
\begin{equation}
g_{\ell} = U(1,p) f_{\ell}, \quad \ell \in \bbN,
\end{equation}
then
\begin{equation}
(F(- i \nabla) g_{\ell})_{r,s} = \begin{cases} 0, & 2 \leq r \leq m, \\
\big(f_{\ell}^{\wedge} F_{p,s}\big)^{\vee}, & r=1,   \end{cases}
\end{equation}
and hence 
\begin{equation}
\interleave g_{\ell} \interleave_{1,m} = \|f_{\ell}\|_1, \quad \ell \in \bbN,
\end{equation}
and 
\begin{align}
\interleave F(- i \nabla) g_{\ell} \interleave_{1,m} 
& = \sum_{s=1}^m \big\|\big(f_{\ell}^{\wedge} F_{p,s}\big)^{\vee}\big\|_1   \no \\
& \geq \big\|\big(f_{\ell}^{\wedge} F_{p,q}\big)^{\vee}\big\|_1 = \|F_{p,q}(- i \nabla) f_{\ell}\|_1  \no \\
& \underset{\ell \to \infty}{\longrightarrow} (2 \pi)^{-n/2} |\mu_{p,q}|(\bbR^n),  
\end{align}
implying the lower bound in \eqref{4.11}.

To show that this lower bound is best possible it suffices to look at the following example. 
With $\gamma_n \colon \gB_n \to [0,1]$ the standard Gaussian measure on $\bbR^n$,
\begin{equation}
\gamma_n (E) = (2 \pi)^{-n/2} \int_{E} \exp\big(- |x|^2/2\big) \, d^n x, \quad 
E \in \gB_n,     \lb{4.28}
\end{equation} 
introduce the measure $\mu_0 \in \cM(\bbR^n, \bbC^{m \times m})$ via
\begin{equation}
\mu_{0,j,k} (E) = \gamma_n (E) \delta_{j,1} \delta_{k,1}, \quad 1 \leq j,k \leq m, \; E \in \gB_n, 
\end{equation}
and let $F_0 = \mu_0^{\wedge}$. For $f \in L^1(\bbR^n, \bbC^{m \times m})$ with 
$\interleave f \interleave_{1,m} =1$ one obtains 
\begin{align}
\interleave F_0(- i \nabla) f \interleave_{1,m} &= \sum_{j=1}^m \big\|\big(f_{j,1}^{\wedge} 
\gamma_n^{\wedge}\big)^{\vee}\big\|_1    \no \\
& \leq \sum_{j=1}^m \| \gamma_n^{\wedge} (- i \nabla)\|_{\cB(L^1(\bbR^n))} 
\sum_{j=1}^m \|f_{j,1}\|_1   \no \\
& \leq (2 \pi)^{-n/2} \gamma_n(\bbR^n) \|f\|_{1,m}  \no \\
& = (2 \pi)^{-n/2} \gamma_n(\bbR^n)    \no \\ 
& = (2 \pi)^{-n/2} N(\mu_0), 
\end{align}
implying $\|F_0(- i \nabla)\|_{\cB((L^1(\bbR^n, \bbC^{m \times m}), \, 
 \interleave \, \cdot \, \interleave_{1,m}))} \leq (2 \pi)^{-n/2} N(\mu_0)$. 

Turning to the upper bound in \eqref{4.11}, let $\varphi \in L^1(\bbR^n, \bbC^{m \times m})$ with 
$\interleave \varphi \interleave_{1,m} = 1$. Then
\begin{equation}
(F(- i \nabla) \varphi)_{j,k} = \sum_{r=1}^m \big(\varphi_{j,r}^{\wedge} F_{r,k}\big)^{\vee}, 
\quad 1 \leq j,k \leq m. 
\end{equation}
Applying the classical (i.e., scalar-valued) $L^1$-multiplier theorem once more, one estimates, 
\begin{align}
\interleave F(- i \nabla) \varphi \interleave_{1,m} &= \sum_{j,k=1}^m \|(F(- i \nabla) \varphi)_{j,k}\|_1 
\no \\
& \leq \sum_{j,k,r=1}^m \big\|\big(\varphi_{j,r}^{\wedge} F_{r,k}\big)^{\vee}\big\|_1   \no \\
& = \sum_{j,k,r=1}^m \|F_{r,k}(- i \nabla) \varphi_{j,r}\|_1  \no \\
& \leq \sum_{j,k,r=1}^m \|F_{r,k}(- i \nabla) \|_{\cB(L^1(\bbR^n))} \|\varphi_{j,r}\|_1    \no \\
& = (2 \pi)^{-n/2} \sum_{j,k,r=1}^m |\mu_{r,k}|(\bbR^n) \|\varphi_{j,r}\|_1   \no \\
& \leq (2 \pi)^{-n/2} N(\mu) \sum_{k=1}^m \sum_{j,r=1}^m \|\varphi_{j,r}\|_1    \no \\
& = (2 \pi)^{-n/2} N(\mu) \, m \interleave \varphi \interleave_{1,m}     \no \\
& = (2 \pi)^{-n/2} N(\mu) \, m.
\end{align}

To demonstrate that this upper bound is best possible, we once more employ the Gaussian measure 
\eqref{4.28} on $\bbR^n$ and hence introduce the measure 
$\mu_1 \in \cM(\bbR^n, \bbC^{m \times m})$ via
\begin{equation}
\mu_{1, j,k} (E) = \gamma_n(E), \quad 1 \leq j,k \leq m, \; E \in \gB_n,
\end{equation}
and let $F_1 = \mu_1^{\wedge}$, such that $F_{1,j,k} = \gamma_n^{\wedge}$, $1 \leq j,k \leq m$. 
Applying the classical multiplier theorem again, one obtains
\begin{equation}
\|F_{1,j,k}(- i \nabla)\|_{\cB(L^1(\bbR^n))} = \gamma_n(\bbR^n) = |\gamma_n|(\bbR^n) =1. 
\end{equation}
Thus, there exists a sequence $\{f_{\ell}\}_{\ell \in \bbN}$ in $L^1(\bbR^n)$ with $\|f_{\ell}\|_1 = 1$, 
$\ell \in \bbN$ such that for all $r,s \in \{1,\dots,m\}$ 
\begin{equation}
\lim_{\ell \to \infty} \|\gamma_n^{\wedge} (- i \nabla) f_{\ell}\|_1 = 
\lim_{\ell \to \infty} \|F_{1,r,s}(- i \nabla) f_{\ell}\|_1 = 1. 
\end{equation}
Let $\varphi_{\ell} \in L^1(\bbR^n, \bbC^{m \times m})$, $\ell \in \bbN$, be defined via 
\begin{equation}
\varphi_{\ell,j,k} = m^{-2} f_{\ell}, \quad \ell \in \bbN, \; 1 \leq j,k \leq m.  
\end{equation}
Then
\begin{equation}
\interleave \varphi_{\ell} \interleave_{1,m} = \sum_{j,k=1}^m \|\varphi_{\ell,j,k}\|_1 
= \sum_{j,k=1}^m m^{-2} \|f_{\ell}\|_1 = 1, \quad \ell \in \bbN. 
\end{equation}
Consequently, 
\begin{align}
\begin{split} 
(F_1(- i \nabla) \varphi_{\ell})_{j,k} 
= \sum_{r=1}^m \big(\varphi_{\ell,j,r}^{\wedge} F_{1,r,k}\big)^{\vee} 
= \sum_{r=1}^m m^{-2} F_{1,r,k}(- i \nabla) f_{\ell},&   \\
\ell \in \bbN, \; 1 \leq j,k \leq m,& 
\end{split} 
\end{align}
and thus,
\begin{align}
\interleave F_1(- i \nabla) \varphi_{\ell} \interleave_{1,m} &= 
\sum_{j,k=1}^m \|(F_1(- i \nabla) \varphi_{\ell})_{j,k}\|_1    \no \\
&= \sum_{j,k=1}^m \bigg\| \sum_{r=1}^m m^{-2} F_{1,r,k}(- i \nabla) f_{\ell}\bigg\|_1    \no \\
&= \sum_{j,k=1}^m m^{-1} \big\|\gamma_n^{\wedge} (- i \nabla) f_{\ell}\big\|_1    \no \\ 
& \underset{\ell \to \infty}{\longrightarrow} (2 \pi)^{-n/2} m = (2 \pi)^{-n/2} m \, N(\mu_1). 
\end{align} 
\end{proof}

Alternatively, one can prove the equivalence of items $(i)$ and $(ii)$ in Theorem \ref{t4.3} 
using \eqref{3.14}, Proposition \ref{p3.3}\,$(ii)$, and Lemma \ref{l3.4}. 

\begin{remark} \lb{r4.3A}
$(i)$ We stress once more that the equivalence of items $(i)$ and $(ii)$ in Theorem \ref{t4.3} was 
proved by Gaudry, Jefferies, and Ricker \cite[Proposition~3.15 and Corollary~3.20]{GJR00} in the 
infinite-dimensional context. For completeness we decided to present a rather elementary and straightforward proof. The bounds \eqref{4.11} appear to be new. \\
$(ii)$ In the special case $m=1$, the upper and lower bound in \eqref{4.11} coincide and hence  
reduce to the classical result $\|F(- i \nabla)\|_{\cB(L^1(\bbR^n))} = (2 \pi)^{-n/2} \|\mu\|$. 
\hfill $\diamond$ 
\end{remark}

Next, we also present the $L^2$-analog of the multiplier Theorem \ref{t4.3} (see, e.g., 
\cite[Theorem~2.5.10]{Gr08}, \cite[p.~28]{St86}, \cite[p.~28, 29]{SW90} for the classical version 
where $m=1$). An infinite-dimensional version of this result appeared in Gaudry, Jefferies, and 
Ricker \cite[Lemma~2.5 and Proposition~2.8]{GJR00}. For completeness, we present an elementary proof in the matrix-valued case (deferring the proof of \eqref{4.41A} to Appendix \ref{sB}) and add the estimates \eqref{4.41a} which appear to be new in this context. 

We recall the definition of $\interleave \cdot \interleave_{2,m}$ in \eqref{3.19a} and $\interleave \cdot \interleave_{\infty,m}$ in \eqref{3.21a}.

\begin{theorem} \lb{t4.4} 
Assume that $F \colon \bbR^n \to \bbC^{m \times m}$ is measurable such that 
$f^{\wedge} F \in L^2(\bbR^n, \bbC^{m \times m})$, $f \in C_0^{\infty}(\bbR^n, \bbC^{m \times m})$ 
and define 
\begin{equation}
F(- i \nabla) \colon \begin{cases} C_0^{\infty}(\bbR^n, \bbC^{m \times m}) \to 
L^2(\bbR^n, \bbC^{m \times m}),  \\ f \mapsto F(- i \nabla) f = 
\big(f^{\wedge} F\big)^{\vee}. \end{cases}
\end{equation}
Then the following conditions $(i)$ and $(ii)$ are equivalent: \\[1mm] 
$(i)$ $F(- i \nabla)|_{C_0^{\infty}(\bbR^n, \bbC^{m \times m})}$ can be extended to a bounded operator 
$($denoted by the same symbol, for simplicity\,$)$ 
$F(- i \nabla) \in \cB\big(L^2(\bbR^n, \bbC^{m \times m})\big)$. \\[1mm]
$(ii)$ $F \in L^{\infty}(\bbR^n, \bbC^{m \times m})$. \\[1mm]
In addition, if one of conditions $(i)$ or $(ii)$ holds, then 
\begin{equation}
\|F(- i \nabla)\|_{\cB(L^2(\bbR^n, \bbC^{m \times m}_{\rm HS}))} 
= {\rm ess.sup}_{x \in \bbR^n} \|F(x)\|_{\cB(\bbC^m)} = \|F\|_{\infty,m},    \lb{4.41A}
\end{equation} 
moreover, one then also has  
\begin{equation}
\interleave F \interleave_{\infty,m} \leq 
\|F(- i \nabla)\|_{\cB((L^2(\bbR^n, \bbC^{m \times m}), \, 
\interleave \, \cdot \, \interleave_{2,m}))} \leq 
m \interleave F \interleave_{\infty,m}.   \lb{4.41a} 
\end{equation}
Both estimates in \eqref{4.41a} are sharp. 
\end{theorem}
\begin{proof}
Assume that condition $(i)$ holds. We recall the definitions of 
$I(j,k)$, $U(j,k)$, $D(j,k)$, and $P(p,q,j,k)$ as in \eqref{4.15}--\eqref{4.18}, 
with $L^1$ replaced by $L^2$. Then as in \eqref{4.19}, 
$P(1,k,1,j) f = \big(f^{\wedge} F_{j,k}\big)^{\vee}$, $f \in L^2(\bbR^n)$, and hence the linear operator 
$L^2(\bbR^n) \ni g \mapsto \big(g^{\wedge} F_{j,k}\big)^{\vee} \in L^2(\bbR^n)$ 
is bounded, that is, $F_{j,k}$ is an $L^2(\bbR^n)$-multiplier. By the classical $L^2$-multiplier theorem, $F_{j,k} \in L^{\infty}(\bbR^n)$, $1 \leq j, k \leq m$, that is, 
$F \in L^{\infty}(\bbR^n, \bbC^{m \times m})$, and hence condition $(ii)$ holds. 

The bound \eqref{4.41A} has been proved in \cite[Lemma~2.5]{GJR00} in the infinite-dimensional context; for completeness we rederive it in the present matrix-valued case in Appendix \ref{sB}. Clearly, the bound \eqref{4.41A} also shows that condition $(ii)$ implies $(i)$. 
   
Next we turn to the lower bound in \eqref{4.41a}. Coose $p, q \in \{1,\dots,m\}$ such that 
\begin{equation} 
\interleave F \interleave_{\infty,m} = \|F_{p,q}\|_{\infty}. 
\end{equation}
Then the classical $L^2$-multiplier theorem (for $m=1$) implies that 
\begin{equation}
\|F_{p,q}(- i \nabla)\|_{\cB(L^2(\bbR^n))} = \|F_{p,q}\|_{\infty}. 
\end{equation}
Thus, there exists a sequence $\{f_{\ell}\}_{\ell \in \bbN}$ in $L^2(\bbR^n)$ with 
$\|f_{\ell}\|_2 = 1$, $\ell \in \bbN$, such that 
\begin{equation}
\lim_{\ell \to \infty} \|F_{p,q}(- i \nabla) f_{\ell}\|_2 = \|F_{p,q}\|_{\infty}. 
\end{equation}
Introducing (cf.\ \eqref{4.16})
\begin{equation}
g_{\ell} = U(1,p) f_{\ell}, \quad \ell \in \bbN,
\end{equation}
then
\begin{equation}
(F(- i \nabla) g_{\ell})_{r,s} = \begin{cases} 0, & 2 \leq r \leq m, \\
\big(f_{\ell}^{\wedge} F_{p,s}\big)^{\vee}, & r=1,   \end{cases}
\end{equation}
and hence 
\begin{equation}
\interleave g_{\ell} \interleave_{2,m} = \|f_{\ell}\|_2 = 1, \quad \ell \in \bbN,
\end{equation}
and 
\begin{align}
\interleave F(- i \nabla) g_{\ell} \interleave_{2,m} 
& = \sum_{s=1}^m \big\|\big(f_{\ell}^{\wedge} F_{p,s}\big)^{\vee}\big\|_2   \no \\
& \geq \big\|\big(f_{\ell}^{\wedge} F_{p,q}\big)^{\vee}\big\|_2 = \|F_{p,q}(- i \nabla) f_{\ell}\|_2  \no \\
& \underset{\ell \to \infty}{\longrightarrow} \|F_{p,q}\|_{\infty},  
\end{align}
implying the lower bound in \eqref{4.41a}.

To show that this lower bound is best possible it suffices to look at the following example. Let 
\begin{equation}
F_{0,j,k} = \delta_{j,1} \delta_{k,1}, \quad 1 \leq j, k, \leq m.
\end{equation}
For $f \in L^2(\bbR^n, \bbC^{m \times m})$ with 
$\interleave f \interleave_{2,m} =1$ one obtains 
\begin{align}
\interleave F_0(- i \nabla) f \interleave_{2,m} 
&= \sum_{j=1}^m \big\|\big(f_{j,1}^{\wedge}\big)^{\vee}\big\|_2 
=  \sum_{j=1}^m \|f_{j,1}\|_2     \no \\
& \leq \|f\|_{2,m} =1    \no \\
& = \interleave F_0 \interleave_{\infty,m},
\end{align}
implying $\|F_0(- i \nabla)\|_{\cB((L^2(\bbR^n, \bbC^{m \times m}), \, 
\interleave \, \cdot \, \interleave_{2,m}))} 
\leq \interleave F_0 \interleave_{\infty,m}$. 

Turning to the upper bound in \eqref{4.41a}, let $\varphi \in L^2(\bbR^n, \bbC^{m \times m})$ with 
$\interleave \varphi \interleave_{2,m} = 1$. Then
\begin{equation}
(F(- i \nabla) \varphi)_{j,k} = \sum_{r=1}^m \big(\varphi_{j,r}^{\wedge} F_{r,k}\big)^{\vee}, 
\quad 1 \leq j,k \leq m. 
\end{equation}
Applying the classical $L^2$-multiplier theorem once more, one estimates, 
\begin{align}
\interleave F(- i \nabla) \varphi \interleave_{2,m} 
&= \sum_{j,k=1}^m \|(F(- i \nabla) \varphi)_{j,k}\|_2     \no \\
& = \sum_{j,k=1}^m \bigg\|\sum_{r=1}^m \big(\varphi_{j,r}^{\wedge} F_{r,k}\big)^{\vee}\bigg\|_2   \no \\
& \leq \sum_{j,k,r=1}^m \|F_{r,k}(- i \nabla) \varphi_{j,r}\|_2  \no \\
& \leq \sum_{j,k,r=1}^m \|F_{r,k}(- i \nabla) \|_{\cB(L^2(\bbR^n))} \|\varphi_{j,r}\|_2    \no \\
& = \sum_{j,k,r=1}^m \|F_{r,k}\|_{\infty} \|\varphi_{j,r}\|_2   \no \\
& \leq \interleave F \interleave_{\infty,m} \sum_{k=1}^m 
\sum_{j,r=1}^m \|\varphi_{j,r}\|_2    \no \\
& = m \interleave F \interleave_{\infty,m} \interleave \varphi \interleave_{2,m}     \no \\
& = m \interleave F \interleave_{\infty,m}. 
\end{align}

To demonstrate that this upper bound is best possible, we introduce 
$F_1 \in L^{\infty}(\bbR^n, \bbC^{m \times m})$ by 
\begin{equation}
F_{1,j,k} =1, \quad 1 \leq j,k \leq m.
\end{equation}
Let $f \in L^2(\bbR^n)$ with $\|f\|_2 = 1$, and introduce 
$\varphi \in L^2(\bbR^n, \bbC^{m \times m})$ via 
\begin{equation}
\varphi_{j,k} = m^{-2} f, \quad 1 \leq j,k \leq m.  
\end{equation}
Then
\begin{equation}
\interleave \varphi \interleave_{2,m} = \sum_{j,k=1}^m \|\varphi_{j,k}\|_2 
= \sum_{j,k=1}^m m^{-2} \|f\|_2 = 1. 
\end{equation}
Consequently, 
\begin{equation}
(F_1(- i \nabla) \varphi)_{j,k} 
= \sum_{r=1}^m \big(\varphi_{j,r}^{\wedge} F_{1,r,k}\big)^{\vee} 
= \sum_{r=1}^m m^{-2} f = m^{-1} f, \quad 1 \leq r,s \leq m, 
\end{equation}
and thus,
\begin{align}
\interleave F_1(- i \nabla) \varphi \interleave_{2,m} &= 
\sum_{j,k=1}^m \|(F_1(- i \nabla) \varphi)_{j,k}\|_2    \no \\
&= \sum_{j,k=1}^m m^{-1} \|f\|_2    \no \\
&= m = m \, \interleave F_1 \interleave_{\infty,m}. 
\end{align}  
\end{proof}

\begin{remark} \lb{r4.4a}
$(i)$ We stress once more that the equivalence of items $(i)$ and $(ii)$ in Theorem \ref{t4.4} (as well 
as the fact \eqref{4.41A}) was proved by Gaudry, Jefferies, and Ricker 
\cite[Lemma~2.5 and Proposition~2.8]{GJR00} in the infinite-dimensional context (we also refer 
to \cite{Pe06} for related results). For completeness we again decided to present a rather elementary and straightforward proof. The bounds \eqref{4.41a} appear to be new. \\
$(ii)$ In the special case $m=1$, the upper and lower bound in \eqref{4.41a} coincide and hence  
reduce to the classical result $\|F(- i \nabla)\|_{\cB(L^2(\bbR^n))} = \|F\|_{\infty}$. 
\hfill$\diamond$
\end{remark}

Next, we provide a matrix-valued extension of a part of Schoenberg's Theorem \cite[Proposition~4.4]{SSV12} (cf.\ Theorem \ref{t1.3}). To be precise, we will show that condition $(i)$ implies 
condition $(iii)$ in Schoenberg's Theorem \ref{t1.3} in the matrix-valued context: 

\begin{theorem} \lb{t4.5} 
Let $F \colon \bbR^n \to \bbC^{m \times m}$ and suppose that $F$ is conditionally positive 
semidefinite and $F(0) \leq 0$. Then for all $N \in \bbN$, $x_p \in \bbR^n$, $1 \leq p \leq N$, 
the block matrix 
$\{F(x_p - x_q) - F(x_p) - F(x_q)^*\}_{1 \leq p,q \leq N} \in \bbC^{mN \times mN}$ is positive semidefinite. 
\end{theorem}
\begin{proof}
Let $x_p \in \bbR^n$, $c_p \in \bbC^m$, $1 \leq p \leq N$. Abbreviating 
$c_0 := - \sum_{p=1}^N c_p$, and $c_p = (c_{p,1}, \dots ,c_{p,m})^\top$, $0 \leq p \leq N$, then 
\begin{equation}
\sum_{p=0}^N \sum_{j=1}^m c_{p,j} = 0.
\end{equation}
In addition, put $x_0 =0 \in \bbR^n$. Then by Lemma \ref{l2.5}\,$(iii)$ one obtains 
\begin{align}
0 & \leq \sum_{p,q = 0}^N (c_p, F(x_p - x_q) c_q)_{\bbC^m}     \no \\ 
& = (c_0, F(0) c_0)_{\bbC^m} + \sum_{p=1}^N (c_p, F(x_p) c_0)_{\bbC^m} 
+ \sum_{q=1}^N (c_0, F(-x_q) c_q)_{\bbC^m}   \no \\
& \quad + \sum_{p,q=1}^N (c_p, F(x_p - x_q) c_q)_{\bbC^m}    \no \\
& = (c_0, F(0) c_0)_{\bbC^m} + \sum_{p=1}^N (c_p, F(-x_p)^* c_0)_{\bbC^m} 
+ \sum_{q=1}^N (c_0, F(-x_q) c_q)_{\bbC^m}   \no \\
& \quad + \sum_{p,q=1}^N (c_p, F(x_p - x_q) c_q)_{\bbC^m}    \no \\
& = (c_0, F(0) c_0)_{\bbC^m} - \sum_{p,q=1}^N (c_p, F(-x_p)^* c_q)_{\bbC^m} 
- \sum_{p,q=1}^N (c_p, F(-x_q) c_q)_{\bbC^m}   \no \\
& \quad + \sum_{p,q=1}^N (c_p, F(x_p - x_q) c_q)_{\bbC^m}. 
\end{align}
Since $x_p \in \bbR^n$, $1 \leq p \leq N$, were arbitrary, replacing $x_p$ by $-x_p$, $1 \leq p \leq N$, 
implies 
\begin{equation}
0 \leq - (c_0, F(0) c_0)_{\bbC^m} 
\leq \sum_{p.q=1}^N (c_p, [F(x_q - x_p) - F(x_q) - F(x_p)^*] c_q)_{\bbC^m},
\end{equation}
completing the proof.
\end{proof}

Combining Theorems \ref{t2.6}, \ref{t2.7}, and \ref{t4.5}, one obtains the following matrix variant of 
Schoenberg's Theorem \ref{t1.3}:

\begin{theorem} \lb{t4.5a}
Let $F \colon \bbR^n \to \bbC^{m \times m}$. Then the following conditions $(i)$ and $(ii)$ are 
equivalent: \\[1mm] 
$(i)$ $F$ is conditionally positive semidefinite.
\\[1mm]
$(ii)$ For all $t > 0$, $\exp_H(t F)$ is positive semidefinite. \\[1mm]
If one of conditions $(i)$ or $(ii)$ holds, and if $F(0) \leq 0$, then the following condition $(iii)$ 
holds: \\[1mm] 
$(iii)$ For all $N \in \bbN$, $x_p \in \bbR^n$, $1 \leq p \leq N$, the block matrix 
$\{F(x_p - x_q) - F(x_p) - F(x_q)^*\}_{1 \leq p,q \leq N} \in \bbC^{mN \times mN}$ is positive semidefinite. 
\end{theorem}

\begin{remark} \lb{r4.5b}
It should be noted that the converse of Theorem \ref{t4.5}, and hence the complete analog of 
Schoenberg's Theorem \ref{t1.3} cannot hold in the matrix-valued context as the following example 
for $m=2$ shows: Choose $n=1$, $m=2$ and 
\begin{equation}
F_0(x) = i x S, \quad S=S^* \in \bbC^{2 \times 2}, \; x \in \bbR,
\end{equation} 
with 
\begin{equation}
S_{j,j} \in \bbR, \; j=1,2, \quad S_{1,2} = \ol{S_{2,1}} = i s, \; s>0. 
\end{equation}
Then 
\begin{equation}
F_0(x_p - x_q) - F_0(x_p) - F_0(x_q)^* = 0, \quad x_p, x_q \in \bbR, \lb{4.72}
\end{equation}
and hence condition $(iii)$ in Theorem \ref{t4.5a} holds for $F_0$ in the special case $n=1$, 
$m = 2$.

Next, pick $x_1, x_2 \in \bbR$, $x_1 > x_2$, then
\begin{equation}
\{F_0(x_p - x_q)\}_{1 \leq p, q \leq 2} = \begin{pmatrix} F_0(0) & F_0(x_1 - x_2) \\
F_0(x_2 - x_1) & F_0(0) \end{pmatrix} = (x_1 - x_2) \begin{pmatrix} 0 & i S \\ - i S & 0 \end{pmatrix}.
\end{equation}
Thus, choosing $c \in \bbR^4$ with $c_1=c_4=0$, $c_3 = - c_2 \neq 0$ one obtains
\begin{equation} 
\sum_{k=1}^4 c_k = 0, \quad 
(c, \{F_0(x_p - x_q)\}_{1 \leq p, q \leq 2} \, c)_{\bbC^4} = - (x_1 - x_2) 2 s c_2^2 < 0, 
\end{equation}
and hence $F_0$ is not conditionally positive semidefinite.  \hfill $\diamond$
\end{remark}

Now we turn to a matrix-valued extension of \cite[Theorem~XIII.52]{RS78} (cf.\ Theorem \ref{t1.4} 
and the subsequent Remark \ref{r4.7}). 

\begin{theorem} \lb{t4.6}
Let $F \in C(\bbR^n, \bbC^{m \times m})$ and suppose there exists $c \in \bbR$ such that
\begin{equation}
\Re(F(x)_{j,k}) \leq c, \quad x \in \bbR^n, \; 1 \leq j,k \leq m.
\end{equation}
Then the following conditions $(i)$--$(iii)$ are equivalent: \\[1mm]
$(i)$ For all $t > 0$, $(\exp_H(t F))(- i \nabla)|_{C_0^{\infty}(\bbR^n, \bbC^{m \times m})}$ extends to a bounded operator $($denoted by the same symbol, for simplicity$)$ 
$(\exp_H(tF))(- i \nabla) 
\in \cB\big(L^1(\bbR^n, \bbC^{m \times m})\big)$ and\footnote{By Lemma \ref{l3.9}\,$(i)$, 
$(\exp_H(tF))(- i \nabla) f \in C_{\infty}(\bbR^n, \bbC^{m \times m})$ for 
$f \in C_0^{\infty}(\bbR^n, \bbC^{m \times m})$, hence the pointwise evaluation 
$((\exp_H(tF))(- i \nabla) f)(x_0)$, $x_0 \in \bbR$, is well-defined. Indeed, if 
$f \in C_0^{\infty}(\bbR^n, \bbC^{m \times m})$, then one concludes that 
$f^{\wedge} \in \cS(\bbR^n, \bbC^{m \times m}) \subset L^1(\bbR^n, \bbC^{m \times m})$. In 
addition, since $\Re(F(\cdot)_{j,k}) \leq c$, $\exp_H (tF) \in L^{\infty}(\bbR^n, \bbC^{m \times m})$ 
and so each entry of 
$f^{\wedge}\exp_H (tF)$ lies in $L^1(\bbR^n)$, and Lemma \ref{l3.9}\,$(i)$ yields that 
$(\exp_H(tF))(- i \nabla) f = \big(f^{\wedge} \exp_H (tF)\big)^{\vee} 
\in C_{\infty}(\bbR^n, \bbC^{m \times m})$.} 
\begin{equation}
\tr_{\bbC^m} \big(((\exp_H(tF))(- i \nabla) f)(0)\big) \geq 0, 
\quad 0 \leq f \in C_0^{\infty}(\bbR^n, \bbC^{m \times m}), \; t > 0.    \lb{4.29}
\end{equation}
$(ii)$ For all $t > 0$, $\exp_H(tF) \colon \bbR^n \to \bbC^{m \times m}$ is positive semidefinite. \\[1mm] 
$(iii)$ $F$ is conditionally positive semidefinite. \\[1mm]
In addition, if one of the conditions $(i)$--$(iii)$ holds, inequality \eqref{4.29} can be replaced by
\begin{equation}
\tr_{\bbC^m} \big(((\exp_H(tF))(- i \nabla) f)(x)\big) \geq 0, 
\quad 0 \leq f \in C_0^{\infty}(\bbR^n, \bbC^{m \times m}), \; x \in \bbR^n, \; t > 0.    \lb{4.29a}
\end{equation}
\end{theorem}
\begin{proof}
Fix $t > 0$. 
Suppose condition $(i)$ holds. Then $\exp_H(tF)$ is an $L^1(\bbR^n, \bbC^{m \times m})$ 
multiplier and hence Theorem \ref{t4.3} guarantees the existence of a measure 
$\mu \in \cM(\bbR^n, \bbC^{m \times m})$ such that $\exp_H(tF) = \mu^{\wedge}$. In addition, 
\begin{align}
& ((\exp_H(tF))(- i \nabla) f)(x) = \big(f^{\wedge} \exp_H(tF)\big)^{\vee}(x)    \lb{4.30} \\
& \quad = (2 \pi)^{-n} \int_{\bbR^n} \int_{\bbR^n} e^{i ((x - \eta) \cdot \xi)} f^{\wedge}(\xi) 
\, d\mu(\eta) \, d^n \xi, \quad f \in C_0^{\infty}(\bbR^n, \bbC^{m \times m}), \; x \in \bbR^n.  \no 
\end{align}
Since $f^{\wedge} \in \cS(\bbR^n, \bbC^{m \times m}) \subset L^1(\bbR^n, \bbC^{m \times m})$, 
one can interchange the order of integration in \eqref{4.30} and obtains 
\begin{align}
& ((\exp_H(tF))(- i \nabla) f)(x) = (2 \pi)^{-n} \int_{\bbR^n} \int_{\bbR^n} 
e^{i ((x - \eta) \cdot \xi)} f^{\wedge}(\xi) \, d^n \xi \, d\mu(\eta)   \no \\
& \quad = (2 \pi)^{-n/2} \int_{\bbR^n} \big(f^{\wedge}\big)^{\vee} (x - \eta) \, d \mu(\eta)   \no \\
& \quad = (2 \pi)^{-n/2} \int_{\bbR^n} f (x - \eta) \, d \mu(\eta)   \no \\
& \quad = (2 \pi)^{-n/2} (T_{\mu}f)(x), \quad f \in C_0^{\infty}(\bbR^n, \bbC^{m \times m}).   \lb{4.31} 
\end{align}
Thus, by condition $(i)$, 
\begin{align}
\begin{split} 
0 \leq \tr_{\bbC^m} \big(((\exp_H(tF))(- i \nabla) f)(0)\big) = (2 \pi)^{-n/2} \tr_{\bbC^m} 
\bigg(\int_{\bbR^n} f(- \eta) \, d \mu(\eta)\bigg),&   \lb{4.32} \\
0 \leq f \in C_0^{\infty}(\bbR^n, \bbC^{m \times m}),&
\end{split} 
\end{align}
and hence 
\begin{equation}
\tr_{\bbC^m} 
\bigg(\int_{\bbR^n} f(x) \, d \mu(x)\bigg) \geq 0, \quad  
0 \leq f \in C_0^{\infty}(\bbR^n, \bbC^{m \times m}).     \lb{4.33} 
\end{equation}
By Lemma \ref{l3.12}, \eqref{4.33} extends to 
\begin{equation}
\tr_{\bbC^m} 
\bigg(\int_{\bbR^n} f(x) \, d \mu(x)\bigg) \geq 0, \quad  
0 \leq f \in C_{\infty}(\bbR^n, \bbC^{m \times m}).     \lb{4.34} 
\end{equation}
By the duality result preceding \eqref{3.7}, this implies $\mu \geq 0$. Applying 
Theorems \ref{t4.2} and \ref{t4.3}, 
$\exp_H(tF) = \mu^{\wedge}$ is positive semidefinite and hence condition $(ii)$ holds. 

Conversely, suppose that condition $(ii)$ holds. Then Theorem \ref{t4.3} implies that  
$(\exp_H(tF))(- i \nabla)|_{C_0^{\infty}(\bbR^n, \bbC^{m \times m})}$ extends to 
$(\exp_H(tF))(- i \nabla) \in \cB\big(L^1(\bbR^n, \bbC^{m \times m})\big)$. As in the first part 
of this proof (cf.\ \eqref{4.31}), one infers 
\begin{equation}
(\exp_H(tF))(- i \nabla) f = (2 \pi)^{-n/2} T_{\mu}f, 
\quad f \in C_0^{\infty}(\bbR^n, \bbC^{m \times m}).  
\end{equation}
Thus, 
\begin{align}
& \tr_{\bbC^m} \big(((\exp(tF))(- i \nabla) f)(0)\big) = (2 \pi)^{-n/2} \tr_{\bbC^m} ((T_{\mu}f)(0))  \no \\
& \quad = (2 \pi)^{-n/2} \tr_{\bbC^m} \bigg(\int_{\bbR^n} f(-y) \, d\mu(y)\bigg) \geq 0, \quad 
0 \leq f \in C_0^{\infty}(\bbR^n, \bbC^{m \times m}),
\end{align}
by the duality result preceding \eqref{3.7}. Thus, condition $(i)$ holds. 

The equivalence of conditions $(ii)$ and $(iii)$  is a consequence of Theorems \ref{t2.6} and 
\ref{t2.7}. 

Finally, if one of conditions $(i)$--$(iii)$ holds, then \eqref{4.29a} follows from \eqref{4.29} 
since by Lemma \ref{l3.4}, $(\exp_H(tF))(- i \nabla)$ commutes with translations.
\end{proof}

\begin{remark} \lb{r4.7}
In the classical case where $m=1$, condition $(i)$ in Theorem \ref{t4.6} can be replaced by the 
following equivalent one: \\[1mm]
$(i')$ For all $t>0$, $(\exp(tF))(-i \nabla) \in \cB\big(L^2(\bbR^n)\big)$ is positivity preserving in 
$L^2(\bbR^n)$. \\[1mm]
Thus Theorem \ref{t4.6} resembles Theorem \ref{t1.4} for $m=1$. In this context we note that 
$(\exp(tF))(-i \nabla) = \exp(t F(- i \nabla))$, $t \geq 0$, for $m=1$.

\noindent 
{\it Proof of $(i')$ implies $(i)$.} If condition $(i')$ holds, then 
\begin{align} 
\begin{split}
\tr_{\bbC} \big(((\exp(tF))(- i \nabla) f)(x)\big) = ((\exp(tF))(- i \nabla) f)(x) \geq 0,& \\ 
\quad 0 \leq f \in C_0^{\infty}(\bbR^n), \; x \in \bbR^n, \; t > 0,&    \lb{4.37a}
\end{split} 
\end{align}
(see also the footnote accompanying Theorem \ref{t4.6}). In particular,
\begin{equation}
\tr_{\bbC} \big(((\exp(tF))(- i \nabla) f)(0)\big) = ((\exp(tF))(- i \nabla) f)(0) \geq 0, 
\quad t > 0, 
\end{equation}
under the assumptions in \eqref{4.37a}.~Since $(\exp(t F))(- i \nabla)$ is positivity preserving, 
Corollary \ref{c3.10} guarantees the existence of a scalar-valued, nonnegative, finite measure $\mu$ 
on $\bbR$ such that 
\begin{equation}
(\exp(t F)) (- i \nabla) = T_{\mu}, \quad t >0.
\end{equation}
Thus, the estimate \eqref{3.14} for $p=1$ yields that $(\exp(t F))(- i \nabla)|_{C_0^{\infty}(\bbR^n)}$ 
extends to a bounded operator $(\exp(t F))(- i \nabla) \in \cB\big(L^1(\bbR^n)\big)$, implying 
condition $(i)$. \\
{\it Proof of $(i)$ implies $(i')$.} If condition $(i)$ holds, then 
\begin{align}
\begin{split} 
((\exp(tF))(- i \nabla) f)(0) = \tr_{\bbC} \big(((\exp(tF))(- i \nabla) f)(0)\big) \geq 0,&  \\
0 \leq f \in C_0^{\infty}(\bbR^n), \; t > 0.& 
\end{split}
\end{align}
By Lemma \ref{l3.4}, this yields 
\begin{align}
\begin{split} 
((\exp(tF))(- i \nabla) f)(x) = \tr_{\bbC} \big(((\exp(tF))(- i \nabla) f)(x)\big) \geq 0,&  \\
0 \leq f \in C_0^{\infty}(\bbR^n), \; x \in \bbR^n, \; t > 0.& 
\end{split}
\end{align}
Since $\{f \in C_0(\bbR^n) \, | \, f \geq 0\}$ is dense in $\big\{f \in L^2(\bbR^n) \, \big| \, f \geq 0\big\}$, 
one concludes that $(\exp(t F))(- i \nabla) \in \cB\big(L^2(\bbR^n)\big)$ is positivity preserving, 
that is, condition $(i')$ holds. ${}$ \hfill $\diamond$
\end{remark}

Next, we will show that the analog of condition $(i')$ in Remark \ref{r4.7} for $m=1$, 
with $\exp(\cdot)$ accordingly replaced by $\exp_H(\cdot)$, {\it cannot} hold for 
$m \geq 2$. We start with two preliminaries:

\begin{lemma} \lb{l4.10} 
Let $F \in C(\bbR^n, \bbC^{m \times m})$ be conditionally positive semidefinite and suppose there exists $c \in \bbR$ such that
\begin{equation}
\Re(F(x)_{j,k}) \leq c, \quad x \in \bbR^n, \; 1 \leq j,k \leq m.
\end{equation}
By Theorem \ref{t4.6}, for all $t>0$, $\exp_H(t F) \colon \bbR^n \to \bbC^{m \times m}$ 
is positive semidefinite, and hence by Theorem \ref{t4.2}, there exists a nonnegative finite measure $\mu_t \in \cM(\bbR^n, \bbC^{m \times m})$, $t>0$, such that 
\begin{equation}
\exp_H(t F)(x) = \mu_t^{\wedge}(x), \quad x \in \bbR^n, \; t>0.    \lb{4.45} 
\end{equation} 
Then, 
\begin{equation}
\mu_{t,j,k} (\bbR^n) \neq 0, \quad 1 \leq j,k \leq m, \; t>0.   \lb{4.46} 
\end{equation}
Thus, for all $t>0$, there exists $R_t > 0$, such that 
\begin{equation}
\mu_{t,j,k}\big(\ol{B(0, R_t)}\big) \neq 0, \quad 1 \leq j,k \leq m.    \lb{4.47} 
\end{equation}
\end{lemma}
\begin{proof}
Since $\exp_H(t F)(x) = \mu_t^{\wedge}(x)$, $x \in \bbR^n$, $t>0$, one concludes that 
\begin{align}
\begin{split} 
0 \neq \exp(t F(0)_{j,k}) &= \exp_H(t F)_{j,k}(0) 
= (2 \pi)^{-n/2} \bigg(\int_{\bbR^n} d\mu_t(x)\bigg)_{j,k}    \\
& = (2 \pi)^{-n/2} \mu_{t,j,k}(\bbR^n), \quad 1 \leq j,k \leq m, 
\end{split} 
\end{align}
and hence \eqref{4.46} holds. Since $\mu_t$ is nonnegative, 
$\mu_t\big(\ol{B(0,R)}\big) \uparrow \mu_t(\bbR^n)$ as $R \to \infty$, thus,
\begin{equation}
\mu_{t,j,k}\big(\ol{B(0,R)}\big) \underset{R \to \infty}{\longrightarrow} \mu_{t,j,k}(\bbR^n), 
\quad 1 \leq j,k \leq m,
\end{equation}
implying \eqref{4.47}. 
\end{proof}

\begin{lemma} \lb{l4.11}
Let $D \in \bbC^{m \times m}$, with $m \in \bbN$, $m \geq 2$, be a strictly positive diagonal matrix with 
\begin{equation}
D_{j,k} = d_j \delta_{j,k}, \quad d_j > 0, \quad 1 \leq j,k \leq m, \quad d_1 \neq d_2,  
\end{equation}
and let $S = S^* \in \bbC^{m \times m}$ be self-adjoint with $S_{1,2} \neq 0$. Then 
$DS$ is not self-adjoint in $\bbC^{m \times m}$. 
\end{lemma}
\begin{proof}
This is clear from $(DS)_{1,2} = d_1 S_{1,2}$ and 
$\ol{(DS)_{2,1}} = d_2 \ol{S_{2,1}} = d_2 S_{1,2}$.
\end{proof}

\begin{theorem} \lb{t4.12}
Let $F \in C(\bbR^n, \bbC^{m \times m})$, $m\geq 2$, be conditionally positive semidefinite and suppose there exists $c \in \bbR$ such that
\begin{equation}
\Re(F(x)_{j,k}) \leq c, \quad x \in \bbR^n, \; 1 \leq j,k \leq m.
\end{equation}
Then for all $t>0$,
\begin{equation}
(\exp_H(t F))(- i \nabla) \in \cB\big(L^2(\bbR^n, \bbC^{m \times m})\big) \, 
\text{ is {\bf not}~positivity preserving.}
\end{equation}
\end{theorem}
\begin{proof}
Fix $t>0$ and let $\mu_t$ and $R_t$ as in Lemma \ref{l4.10}, and $D \in \bbC^{m \times m}$ be the strictly positive diagonal matrix in Lemma \ref{l4.11}. For sufficiently small $\varepsilon > 0$ we introduce 
\begin{equation}
h_{\varepsilon} \in C^{\infty}([0,\infty)), \quad h_{\varepsilon}(r) 
= \begin{cases} 1, & r \in [0, R_t], \\
0, & r \in [R_t + \varepsilon, \infty), \end{cases} 
\end{equation}
and 
\begin{equation}
0 \leq g_{\varepsilon} \in C_0^{\infty}(\bbR^n, \bbC^{m \times m}), \quad 
g_{\varepsilon} (x) = h_{\varepsilon}(|x|) D, \; x \in \bbR^n. 
\end{equation}
Then, 
\begin{equation}
(((\exp_H(t F))(- i \nabla))f)(x) = (2 \pi)^{-n/2} (T_{\mu_t} f)(x), \quad 
f \in C_0^{\infty}(\bbR^n, \bbC^{m \times m}), \; x \in \bbR^n,
\end{equation}
by the part of the proof of Theorem \ref{t4.3} that condition $(ii)$ implies $(i)$. Thus, 
\begin{align}
& (((\exp_H(t F))(- i \nabla))g_{\varepsilon})(0) = (2 \pi)^{-n/2} (T_{\mu_t} g_{\varepsilon})(0) \no \\
& \quad = (2 \pi)^{-n/2} \int_{\bbR^n} g_{\varepsilon}(-y) \, d\mu_t(y)    \no \\
& \quad = (2 \pi)^{-n/2} \int_{\bbR^n} g_{\varepsilon}(y) \, d\mu_t(y)    \no \\
& \quad = (2 \pi)^{-n/2} \bigg[\int_{\ol{B_n(0,R_t)}} g_{\varepsilon}(y) \, d\mu_t(y)  
+ \int_{B_n(0,R_t + \varepsilon) \backslash \ol{B_n(0,R_t)}} g_{\varepsilon}(y) \, d\mu_t(y)\bigg] 
\no \\
& \quad = (2 \pi)^{-n/2} D \mu_t\big(\ol{B_n(0,R_t)}\big) + (2 \pi)^{-n/2} 
\int_{B_n(0,R_t + \varepsilon) \backslash \ol{B_n(0,R_t)}} g_{\varepsilon}(y) \, d\mu_t(y). 
\end{align} 
By estimate \eqref{3.6}, 
\begin{align}
& \bigg\|\int_{B_n(0,R_t + \varepsilon) \backslash \ol{B_n(0,R_t)}} g_{\varepsilon}(y) \, d\mu_t(y)
\bigg\|    \no \\
& \quad \leq \int_{B_n(0,R_t + \varepsilon) \backslash \ol{B_n(0,R_t)}} 
\|g_{\varepsilon}(y)\|_{\cB(\bbC^m)} \, d|\mu_t|(y)    \no \\
& \quad \leq \int_{B_n(0,R_t + \varepsilon) \backslash \ol{B_n(0,R_t)}} 
\|D\|_{\cB(\bbC^m)} \, d|\mu_t|(y)    \no \\
& \quad \underset{\varepsilon \downarrow 0}{\longrightarrow} 0.
\end{align} 
Using the fact that 
\begin{equation}
\cN_m = \bbC^{m \times m} \backslash \{ A^* A \in \bbC^{m \times m} \, | \, A \in \bbC^{m \times m}\} 
\, \text{ is open in $\bbC^{m \times m}$}   \lb{4.58} 
\end{equation}
(since nonnegative $m \times m$ matrices form a closed cone in $\bbC^{m \times m}$), 
employing  
\begin{equation}
D \mu_t\big(\ol{B_n(0,R_t)}\big) \in \cN_m, 
\end{equation}
applying Lemma \ref{l4.11} with $S =  \mu_t\big(\ol{B_n(0,R_t)}\big)$, and utilizing 
\begin{equation}
((\exp_H(t F))(- i \nabla))(g_{\varepsilon}) \in L^2(\bbR^n, \bbC^{m \times m}) \cap 
C_{\infty}(\bbR^n, \bbC^{m \times m})
\end{equation}
employing Lemma \ref{l3.9}\,$(i)$, one concludes that for all sufficiently small $\varepsilon > 0$, 
$(\exp_H(t F) (- i \nabla) g_{\varepsilon})(0)$ is not nonnegative. Thus, for all sufficiently small 
$\varepsilon > 0$, there exists $\eta(\varepsilon) > 0$, such that 
$(\exp_H(t F) (- i \nabla) g_{\varepsilon})(x)$, $x \in B_n(0, \eta(\varepsilon))$, is not nonnegative. Since 
$g_{\varepsilon} \geq 0$, this completes the proof. 
\end{proof}

Thus, unlike the classical case $m=1$ discussed in Remark \ref{r4.7}, the straightforward extension of Theorem \ref{t1.4} replacing its condition $(i)$ by  \\[1mm] 
$(i')$ For all $t > 0$, $(\exp_H)(t F))(-i \nabla)$ is positivity preserving \\[1mm] 
cannot hold in the matrix-valued context, $m \geq 2$. 

Finally, we derive the bound \eqref{1.7} in the matrix-valued context following 
\cite[Lemma~3.6.22]{Ja01}. First, we recall the following fact:

\begin{proposition} [{\cite[p.~112]{Ho90}}] \lb{p2.8}
Let $0 \leq M_{\ell} \in \bbC^{m_{\ell} \times m_{\ell}}$, $m_{\ell} \in \bbN$, $\ell =1,2$ $($i.e., 
$M_{\ell}$, $\ell = 1,2$ are positive semidefinite\,$)$, and $X \in \bbC^{m_1 \times m_2}$.~Introduce 
the block matrix
\begin{equation}
A = \begin{pmatrix} M_1 & X \\ X^* & M_2 \end{pmatrix} \in \bbC^{(m_1 + m_2) \times (m_1 + m_2)}.
\end{equation}
Then $A$ is positive semidefinite $($i.e., $A \geq 0$$)$ if an only if there exists a contraction 
$C \in \bbC^{m_1 \times m_2}$ such that $X = M_1^{1/2} C M_2^{1/2}$.  
\end{proposition}

Here $C$ is viewed as a linear map $C \colon \bbC^{m_2} \to \bbC^{m_1}$, and, according to our convention, we employ the standard Euclidean scalar product and norm on $\bbC^{m_{\ell}}$, 
$\ell = 1,2$.   

Next, we state a preparatory result:

\begin{lemma} \lb{l4.13}
Suppose that $F \in C(\bbR^n, \bbC^{m \times m})$ is conditionally positivie semidefinite with 
$F(0) \leq 0$. Then,
\begin{align}
& 0 \leq F(0) - 2 \Re(F(x)) \leq - 2 \Re(F(x)), \quad x \in \bbR^n,   \lb{4.38} \\
&\|F(0) - 2 \Re(F(x))\|_{\cB(\bbC^m)} \leq 2 \|\Re(F(x))\|_{\cB(\bbC^m)} \leq 
2 \|F(x)\|_{\cB(\bbC^m)}, \quad x \in \bbR^n,    \lb{4.39} \\
& \|F(x-y) - F(x) - F(y)^*\|_{\cB(\bbC^m)} \leq 2 \|F(x)\|_{\cB(\bbC^m)}^{1/2} 
\|F(y)\|_{\cB(\bbC^m)}^{1/2}, \quad x, y \in \bbR^n,    \lb{4.40} \\
& \|F(x+y)\|_{\cB(\bbC^m)}^{1/2} \leq \|F(x)\|_{\cB(\bbC^m)}^{1/2} + 
\|F(y)\|_{\cB(\bbC^m)}^{1/2}, \quad x, y \in \bbR^n.    \lb{4.41} 
\end{align}
\end{lemma}
\begin{proof}
Inequality \eqref{4.38} follows from Theorem \ref{t4.5} and from $F(0) \leq 0$, and \eqref{4.39} is a 
consequence of \eqref{3.37}--\eqref{3.39}.

Next, denote $G(x)= F(0) - F(x) - F(x)^*$, $H(x,y) = F(x-y) - F(x) - F(y)^*$, and 
$K(y) = F(0) - F(y) - F(y)^*$. Applying once more Theorem \ref{t4.5} one infers that
\begin{equation}
0 \leq \begin{pmatrix} G(x) & H(x,y) \\ H(x,y)^* & K(y) \end{pmatrix} \in \bbC^{2m \times 2m}, 
\quad x,y \in \bbR^n. 
\end{equation}
By \eqref{4.38}, 
\begin{equation}
G(x) \geq 0, \quad K(y) \geq 0, \quad x, y \in \bbR^n,
\end{equation}
and hence Proposition \ref{p2.8} guarantees the existence of a linear contraction 
$C(x,y) \in \bbC^{m \times m}$, $x, y \in \bbR^n$, such that 
\begin{equation}
H(x,y) = G(x)^{1/2} C(x,y) K(y)^{1/2}, \quad x, y \in \bbR^n. 
\end{equation}
Thus, \eqref{4.39} yields
\begin{align}
& \|H(x,y)\|_{\cB(\bbC^m)}  \leq \|G(x)\|_{\cB(\bbC^m)}^{1/2} 
\|K(y)\|_{\cB(\bbC^m)}^{1/2} \leq 2 \|F(x)\|_{\cB(\bbC^m)}^{1/2} 
\|F(y)\|_{\cB(\bbC^m)}^{1/2},  \no \\
& \hspace*{9cm}  x, y \in \bbR^n,
\end{align} 
proving \eqref{4.40}. 

By \eqref{4.40} one obtains
\begin{align}
& \|F(x-y)\|_{\cB(\bbC^m)}  - \|F(x)\|_{\cB(\bbC^m)}  - \|F(y)^*\|_{\cB(\bbC^m)}    \no \\
& \quad \leq \|F(x-y)\|_{\cB(\bbC^m)}  - \|F(x) + F(y)^*\|_{\cB(\bbC^m)}   \no \\
& \quad \leq \|F(x-y) - F(x) - F(y)^*\|_{\cB(\bbC^m)}   \no \\
& \quad \leq 2 \|F(x)\|_{\cB(\bbC^m)}^{1/2} \|F(y)\|_{\cB(\bbC^m)}^{1/2}, \quad x, y \in \bbR^n,
\end{align} 
implying 
\begin{equation}
 \|F(x-y)\|_{\cB(\bbC^m)} \leq \big[\|F(x)\|_{\cB(\bbC^m)}^{1/2} + \|F(y)\|_{\cB(\bbC^m)}^{1/2}\big]^2, 
 \quad x, y \in \bbR^n. 
\end{equation}
Replacing $y$ by $-y$ and using $F(-y) = F(y)^*$ yields \eqref{4.41}. 
\end{proof}

\begin{theorem} \lb{t4.14}
Suppose that $F \colon \bbR^n \to \bbC^{m \times m}$ is locally bounded and conditionally positive semidefinite with $F(0) \leq 0$. Then, there exists $C>0$ such that
\begin{equation}
 \|F(x)\|_{\cB(\bbC^m)} \leq C \big[1 + |x|^2\big], \quad x \in \bbR^n. 
\end{equation}
\end{theorem} 
\begin{proof}
By local boundedness of $F$ it suffices to prove the existence of $C' > 0$ such that 
$\|F(x)\|_{\cB(\bbC^m)} \leq C' |x|^2$ for $|x|$ sufficiently large. 
Thus, for $x \in \bbR^n$ with $|x| \geq 2$, let $m(x) \in \bbN$ be the positive integer 
such that $|x| \in [m(x), m(x) +1)$. Then by \eqref{4.41},
\begin{align}
\|F(x)\|_{\cB(\bbC^m)}^{1/2} &= \|F(m(x) (x/m(x)))\|_{\cB(\bbC^m)}^{1/2} \leq m(x) 
\|F((x/m(x)))\|_{\cB(\bbC^m)}^{1/2}   \no \\
& \leq m(x) \big[\sup_{y \in \bbR^n} 
\{\|F(y)\|_{\cB(\bbC^m)} \, | \, 0 \leq |y| \leq 2\}\big]^{1/2}    \no \\ 
& \leq [C']^{1/2} |x|, \quad |x| \geq 2,
\end{align}
where 
\begin{equation}
C' = \sup_{y \in \bbR^n} 
\big\{\|F(y)\|_{\cB(\bbC^m)} \, \big| \, 0 \leq |y| \leq 2\big\}.  
\end{equation}
\end{proof}

We conclude with some elementary examples of conditionally positive semidefinite 
matrix-valued functions\footnote{Part~$(i)$ of Example \ref{e4.17} now replaces the originally published version which had a mistake.} on $\bbR^n$. 

\begin{example} \lb{e4.17}  ${}$ \\
$(i)$ Fix $y_0 \in \bbR^n \backslash \{0\}$ and $a, b, c \in (0,\infty)$, with 
$ac \geq b^2$. Then $F_2 \colon \bbR^n \to \bbC^{2 \times 2}$ defined via
\begin{equation}
F_2(x) = \begin{pmatrix} -i (x \cdot y_0) + \ln (a) & -i (x \cdot y_0) + \ln(b) \\ 
-i (x \cdot y_0) + \ln(b)  & -i (x \cdot y_0) + \ln(c) \end{pmatrix}, 
\quad x \in \bbR^n,  
\end{equation}
is conditionally positive semidefinite. \\[1mm]
$(ii)$ Suppose that $G_0 \colon \bbR^n \to \bbC$ is conditionally positive semidefinite and 
introduce the constant matrix $H = \{H_{j,k}\}_{1 \leq j,k \leq m} \in \bbC^{m \times m}$ by
\begin{equation}
H_{j,k} = 1, \quad 1 \leq j,k \leq m.
\end{equation}
Then $F_0 \colon \bbR^n \to \bbC^{m \times m}$ defined by 
\begin{equation}
F_0(x) = G_0(x) H, \quad x \in \bbR^n, 
\end{equation}
is conditionally positive semidefinite. 
\end{example}
\begin{proof}
$(i)$ For all $t > 0$, one infers that
\begin{align}
(\exp_{H} (t F_2))_{1,1}(x)  &= a^t e^{-i t (x \cdot y_0)} 
= (2 \pi)^{- n/2} \bigg[(2 \pi)^{n/2} a^t \int_{\bbR^n} e^{- i (x \cdot y)} d \delta_{t y_0}(y) \bigg]   \no \\
&= \big((2 \pi)^{n/2} a^t \delta_{t y_0}\big)^{\wedge} (x),  \\
(\exp_{H} (t F_2))_{2,2}(x) 
&= \big((2 \pi)^{n/2} c^t \delta_{t y_0}\big)^{\wedge} (x),  \\
(\exp_{H} (t F_2))_{1,2}(x) 
&= \big((2 \pi)^{n/2} b^t \delta_{t y_0}\big)^{\wedge} (x) 
= (\exp_{H} (t F_2))_{2,1}(x), 
\end{align}
where $\delta_{x_0}$ denotes the usual Dirac measure supported at 
$x_0 \in \bbR^n$ of unit mass. Next, we introduce 
$\mu_{2,t} \in \cM(\bbR^n, \bbC^{2 \times 2})$) via
\begin{equation}
\mu_{2,t} (E) = (2 \pi)^{n/2} \begin{pmatrix} a^t \delta_{t y_0}(E) 
& b^t \delta_{t y_0}(E) \\  
b^t \delta_{t y_0}(E) & c^t \delta_{t y_0}(E) \end{pmatrix}, \quad t > 0, \; E \in \gB_n. 
\end{equation}
Since for all $E \in \gB_n$, $\mu_{2,t}(E)$, $t > 0$, can only take on the two values,
\begin{equation}
\mu_{2,t} (E) = \begin{cases} 
 \begin{pmatrix} 0 & 0 \\ 0 & 0 \end{pmatrix} \, \text{ if $t y_0 \notin E$,}  \\
(2 \pi)^{n/2} \begin{pmatrix} a^t & b^t \\ b^t & c^t \end{pmatrix} \, 
\text{ if $t y_0 \in E$,} 
\end{cases} \lb{4.117}  
\end{equation}
and the hypothesis $ac \geq b^2$ implies $\Big(\begin{smallmatrix} a^t & b^t \\ b^t & c^t 
\end{smallmatrix}\Big) \geq 0$, $t >0$, $\mu_{2,t}$ is a nonnegative measure (i.e., 
$0 \leq \mu_{2,t} \in \cM(\bbR^n, \bbC^{2 \times 2})$) for all $t > 0$. Thus, 
\begin{equation}
\exp_{H} (t F_2(x)) = \mu_{2,t}^{\wedge}(x), \quad x \in \bbR^n, \; t > 0,   
\end{equation}
and by Theorems \ref{t4.2} and the equivalence of items $(ii)$ and $(iii)$ in 
Theorem \ref{t4.6}, $F_2$ is conditionally positive semidefinite. \\
$(ii)$ Since $G_0$ is assumed to be conditionally positive semidefinite,  
$\exp(t G_0) \colon \bbR^n \to \bbC$ is positive semidefinite for all $t > 0$. So by the classical 
Bochner theorem, for all $t > 0$, there exists a nonnegative scalar-valued measure $\nu_t$ on 
$\bbR^n$ such that 
\begin{equation}
e^{t G_0} = \nu_t^{\wedge}, \quad t > 0.  
\end{equation} 
Introducing 
\begin{equation}
\mu_{0,t}(E) = \nu_t(E) H, \quad E \in \gB_n, \; t > 0,  
\end{equation}
then $\mu_{0,t}$, $t > 0$, is nonnegative and 
\begin{equation}
\exp_{H} (t F_0) = e^{t G_0} H = \nu_t^{\wedge} H = \mu_{0,t}^{\wedge}, \quad t > 0. 
\end{equation}
Thus $F_0$ is conditionally positive semidefinite utilizing once more Theorem \ref{t4.2} and the equivalence of items $(ii)$ and $(iii)$ in Theorem \ref{t4.6}. 
\end{proof}

\appendix
\section{A Counterexample} \lb{sA}
\renewcommand{\theequation}{A.\arabic{equation}}
\renewcommand{\thetheorem}{A.\arabic{theorem}}
\setcounter{theorem}{0} \setcounter{equation}{0}

In this appendix we verify the claim made in Remark \ref{r4.1a}.~For brevity, we construct the 
counterexample for $m=2$, but the construction extends to general $m \in \bbN$, $m \geq 3$. 

Let $\gamma_n \colon \gB_n \to [0,1]$ be the standard Gaussian measure on $\bbR^n$,
\begin{equation}
\gamma_n (E) = (2 \pi)^{-n/2} \int_{E} \exp\big(- |x|^2/2\big) \, d^n x, \quad 
E \in \gB_n, 
\end{equation} 
and introduce 
\begin{equation}
\mu(E) = \gamma_n(E) A, \quad A = \begin{pmatrix} 1 & 0 \\ 0 & 2 \end{pmatrix} \geq 0, \quad E \in \gB_n, \quad F = \mu^{\wedge}, 
\end{equation}
and 
\begin{equation}
M = \begin{pmatrix} 3 & 1 \\ 1 & 3 \end{pmatrix} \geq 0, \, \text{ such that } \, 
M A = \begin{pmatrix} 3 & 2 \\ 1 & 6 \end{pmatrix} 
\end{equation}
is not self-adjoint, let alone positive semidefinite. 

As in the proof of Theorem \ref{t4.3} that condition $(ii)$ implies condition $(i)$, one obtains,
\begin{equation}
(F(- i \nabla) f)(x) = (2 \pi)^{-n/2} (T_{\mu} f)(x), \quad f \in C_0^{\infty}(\bbR^n, \bbC^{2 \times 2}), 
\; x \in \bbR^n. 
\end{equation}
Next, for sufficiently small $\varepsilon > 0$, consider $h_{\varepsilon} \in C_0^{\infty}(\bbR^n)$ 
satisfying,
\begin{align}
0 \leq h_{\varepsilon}(x) \leq 1, \; x \in \bbR^n, \quad 
h_{\varepsilon}(x) = \begin{cases} 1, & x \in \ol{B_n(0,1)}, \\
0, & x \in \bbR^n \backslash B_n(0, 1 + \varepsilon), \end{cases} 
\end{align}
and let 
\begin{equation}
g_{\varepsilon} (x) = h_{\varepsilon}(x) M, \quad x \in \bbR^n. 
\end{equation}
Then, 
\begin{align}
(F(- i \nabla) g_{\varepsilon})(0) &= (2 \pi)^{-n/2} (T_{\mu} g_{\varepsilon})(0)   \no \\
&= (2 \pi)^{-n/2} \int_{\bbR^n} g_{\varepsilon}(-y) \, d\mu(y)    \no \\
&= (2 \pi)^{-n/2} \int_{\bbR^n} g_{\varepsilon}(y) \, d\mu(y)    \no \\
&= (2 \pi)^{-n/2} \int_{B_n(0,1)} g_{\varepsilon}(y) \, d\mu(y)   \no \\ 
& \quad + 
(2 \pi)^{-n/2} \int_{B_n(0,1 + \varepsilon) \backslash B_n(0,1)} g_{\varepsilon}(y) \, d\mu(y)    \no \\
&= (2 \pi)^{-n/2} \gamma_n (B_n(0,1)) MA    \no \\ 
& \quad + 
(2 \pi)^{-n/2} \int_{B_n(0,1 + \varepsilon) \backslash B_n(0,1)} g_{\varepsilon}(y) \, d\mu(y).  
\end{align}
By \eqref{3.6},
\begin{align}
& \bigg\| \int_{B_n(0,1 + \varepsilon) \backslash B_n(0,1)} g_{\varepsilon}(y) \, d\mu(y)
\bigg\|_{\cB(\bbC^m)}     \no \\
& \quad \leq  \int_{B_n(0,1 + \varepsilon) \backslash B_n(0,1)} \|g_{\varepsilon}(y)\|_{\cB(\bbC^m)} 
\, d|\mu|(y)    \no \\
& \quad \leq  \int_{B_n(0,1 + \varepsilon) \backslash B_n(0,1)} \|M\|_{\cB(\bbC^m)} 
\, d|\mu|(y)    \no \\
& \quad = \|M\|_{\cB(\bbC^m)} \|A\|_{\cB(\bbC^m)} 
\gamma_{n}(B_n(0,1 + \varepsilon) \backslash B_n(0,1))    \no \\
& \quad \underset{\varepsilon \downarrow 0}{\longrightarrow} 0. 
\end{align}

Since the set
\begin{equation}
\cN_2 = \bbC^{2 \times 2} \backslash \{ A^* A \in \bbC^{2 \times 2} \, | \, A \in \bbC^{2 \times 2}\}
\end{equation}
is open in $\bbC^{2 \times 2}$ (cf.\ \eqref{4.58}), since 
\begin{equation}
\gamma_n(B_n(0,1)) MA \in \cN_2, 
\end{equation}
and since $F(- i \nabla) g_{\varepsilon} \in L^2(\bbR^n, \bbC^{m \times m}) \cap 
C_{\infty}(\bbR^n, \bbC^{m \times m})$ by Lemma \ref{l3.9}\,$(i)$, for $0 < \varepsilon$ 
sufficiently small, $(F(- i \nabla) g_{\varepsilon})(0)$ is {\it not} positive semidefinite, and thus 
there exists $\delta(\varepsilon) > 0$ such that $(F(- i \nabla) g_{\varepsilon})(x)$ is not positive semidefinite for all $x \in B_n(0,\delta(\varepsilon))$, even though $g_{\varepsilon} \geq 0$, illustrating 
Remark \ref{r4.1a}. 

In the special case where $\mu_{\sigma}(E) = \sigma(E) I_{\bbC^m}$, $E \in \gB_n$, with 
$\sigma \colon \gB_n \to [0, \infty)$ a finite meausure, and $F = \mu_{\sigma}^{\wedge}$, 
$F (- i \nabla) = (2 \pi)^{-n/2} T_{\mu_{\sigma}}$ is of course positivity preserving in $L^2(\bbR^n, \bbC^{m \times m})$.

\section{The Multiplier Norm Equality \eqref{4.41A}} \lb{sB}
\renewcommand{\theequation}{B.\arabic{equation}}
\renewcommand{\thetheorem}{B.\arabic{theorem}}
\setcounter{theorem}{0} \setcounter{equation}{0}

The purpose of this appendix is an elementary and straightforward proof of the multiplier norm 
equality \eqref{4.41A}.

We start with some preliminary observations. First, we will employ the convention that each matrix 
in $\bbC^{m \times m}$ will be identified with a column vector in $\bbC^{m^2}$ in the manner that we list the entries of the matrix from left to right, starting from the 1st row to the $m$-th row. We also recall the possible identifications, 
\begin{equation}
\bbC^{m \times m}_{\rm HS} \simeq \cB_2(\bbC^m) \simeq \bbC^{m^2}, \quad 
\cB(\bbC^{m \times m}_{\rm HS}) \simeq \cB\big(\bbC^{m^2}\big) \simeq \bbC^{m^2 \times m^2},  
\end{equation}
consistently employing the Euclidean norm on $\bbC^m$ and $\bbC^{m^2}$.

In addition, given $A \in \bbC^{m \times m}$ we introduce the linear operator $M_A$ of right multiplication by $A$ on $\bbC^{m \times m}$ via,
\begin{equation}
M_A(B) := BA, \quad B \in \bbC^{m \times m}.   \lb{B.2} 
\end{equation}
Since $M_A$ is a linear operator on $\bbC^{m^2}$, it is representable by a matrix 
$K_A \in \bbC^{m^2 \times m^2}$, and the latter may be described upon inspection as follows:

\begin{lemma} \lb{lB.1}  
$K_A$ is a block matrix with $m^2$ blocks, $m$ blocks across horizontally and $m$ blocks vertically. Each block is an $m \times m$ matrix, the diagonal blocks each equal $A^{\top}$ $($the transpose of $A$$)$, and all off-diagonal blocks equal the zero matrix in $\bbC^{m \times m}$.  
\end{lemma} 

Then one obtains the following result for the operator norm of $M_A$.

\begin{proposition} \lb{pB.2} 
Let $A \in \bbC^{m \times m}$, then 
\begin{equation}
\|M_A\|_{\cB(\bbC^{m^2})} 
= \|K_A\|_{\cB(\bbC^{m^2})} = \|A\|_{\cB(\bbC^m)},
\end{equation}
where, according to our conventions, $\bbC^m$ and $\bbC^{m^2}$ are equipped with the Euclidean norm.
\end{proposition}
\begin{proof}
Let $\{u_j\}_{j \in \bbN}$ be a sequence in $\bbC^m$ such that $\|u_j\|_{\bbC^m} = 1$, $j \in \bbN$, 
and $\lim_{j \to \infty} \big\|A^{\top} u_j\big\|_{\bbC^m} = \big\|A^{\top}\big\|_{\cB(\bbC^m)}$. 
For each $j \in \bbN$, let $v_j \in \bbC^{m^2}$ be the column vector obtained by repeating $u_j$ 
$m$ times down the column, and introduce 
\begin{equation}
\omega_j = m^{-1/2} v_j \in \bbC^{m^2}, \, \text{ such that } \, \|\omega_j\|_{\bbC^{m^2}} =1, \quad 
j \in \bbN. 
\end{equation} 
Then for all $j \in \bbN$, $K_A \omega_j \in \bbC^{m^2}$ is the column vector obtained upon repeating $m^{-1/2} A^{\top} u_j$ $m$ times down the column such that 
\begin{equation}
\|K_A \omega_j\|_{\bbC^{m^2}} = \big\|A^{\top} u_j\big\|_{\bbC^m} 
\underset{j \to \infty}{\longrightarrow} \big\|A^{\top}\big\|_{\cB(\bbC^m)}. 
\end{equation} 
Thus,
\begin{equation}
\|K_A\|_{\cB(\bbC^{m^2})} \geq \big\|A^{\top}\big\|_{\cB(\bbC^m)} = \|A\|_{\cB(\bbC^m)}, 
\quad A \in \bbC^{m \times m}.  
\end{equation}
To prove the opposite inequality we identify 
$\bbC^{m \times m}_{\rm HS} = (\bbC^{m \times m}, \|\, \cdot \,\big\|_{\rm HS})$ 
with $\bbC^{m^2}$ and observe that for all $B \in \bbC^{m \times m} \simeq \bbC^{m^2}$ one has 
\begin{align}
\|M_A(B)\|_{\bbC^{m^2}} &= \|BA\|_{(\bbC^{m \times m}, \, \|\, \cdot \,\|_{\rm HS})} 
= \|BA\|_{\cB_2(\bbC^m)} \leq \|B\|_{\cB_2(\bbC^m)} \|A\|_{\cB(\bbC^m)}     \no \\
& = \|B\|_{\bbC^{m^2}} \|A\|_{\cB(\bbC^m)}, 
\end{align} 
implying
\begin{equation}
\|M_A(B)\|_{\bbC^{m^2}} \leq \|A\|_{\cB(\bbC^m)}. 
\end{equation} 
\end{proof}

At this point we can turn to the principal aim of this appendix: \\[1mm]  
{\it Proof of \eqref{4.41A}}. Let 
\begin{align}
\begin{split} 
& \Phi : \bbR^n \to \cB(\bbC^{m \times m}_{\rm HS}) \simeq \bbC^{m^2 \times m^2} \, 
\text{  be measurable,} \\
& \text{and assume that } \, 
\|\Phi\|_{\infty,m^2} = {\rm ess.sup}_{x \in \bbR^n} \|\Phi(x)\|_{\cB(\bbC^{m \times m}_{\rm HS})} < \infty,  
\lb{B.9} 
\end{split} 
\end{align}
and introduce 
\begin{equation}
S_{\Phi} \colon \begin{cases} L^2(\bbR^n, \bbC^{m \times m}_{\rm HS}) 
\to L^2(\bbR^n, \bbC^{m \times m}_{\rm HS}),     \\
(S_{\Phi} f)^{\wedge}(y) = \Phi(y) f^{\wedge}(y) \, \text{ for a.e.\ $y \in \bbR^n$}. 
\end{cases}  
\end{equation}

\begin{lemma} \lb{lB.3}
Assume \eqref{B.9}. Then
\begin{equation}
\|S_{\Phi}\|_{\cB(L^2(\bbR^n, \, \bbC^{m \times m}_{\rm HS}))} \leq \|\Phi\|_{\infty,m^2}. 
\end{equation}
\end{lemma}
\begin{proof}
Let $f \in L^2(\bbR^n, \bbC^{m \times m}_{\rm HS})$ with 
$\|f\|_{L^2(\bbR^n, \, \bbC^{m \times m}_{\rm HS})} = 1$. Then
\begin{align}
\|S_{\Phi} f\|_{L^2(\bbR^n, \, \bbC^{m \times m}_{\rm HS})} 
& = \big\|(S_{\Phi} f)^{\wedge}\big\|_{L^2(\bbR^n, \, \bbC^{m \times m}_{\rm HS})}  \no \\
& = \bigg(\int_{\bbR^n} \big\|\Phi(y) f^{\wedge}(y)\big\|^2_{\bbC^{m^2}} \, d^nx\bigg)^{1/2}  
\no \\
& \leq \bigg(\int_{\bbR^n} \|\Phi(y)\|^2_{\cB(\bbC^{m \times m}_{\rm HS})} 
\big\|f^{\wedge}(y)\big\|^2_{\bbC^{m^2}} \, d^nx\bigg)^{1/2}    \no \\ 
& \leq \|\Phi\|_{\infty,m^2} \bigg(\int_{\bbR^n} 
\big\|f^{\wedge}(y)\big\|^2_{\bbC^{m^2}} \, d^nx\bigg)^{1/2}    \no \\
& = \|\Phi\|_{\infty,m^2} \big\|f^{\wedge}\big\|_{L^2(\bbR^n, \, \bbC^{m \times m}_{\rm HS})}    \no \\
& = \|\Phi\|_{\infty,m^2} \|f\|_{L^2(\bbR^n, \, \bbC^{m \times m}_{\rm HS})}    \no \\
& = \|\Phi\|_{\infty,m^2}. 
\end{align}
\end{proof}

\begin{lemma} \lb{lB.4}
Assume that $\wti \Phi$ is a simple function, that is, there exist $J \in \bbN$, $a_j \in \bbC$, 
$\Phi_j \in \cB(\bbC^{m \times m}_{\rm HS})$, with $\|\Phi_j\|_{\cB(\bbC^{m \times m}_{\rm HS})} = 1$, and $E_j \in \gB_n$, $1 \leq j \leq J$, such that $\wti \Phi$ is of the type, 
\begin{equation}  
\wti \Phi = \sum_{j=1}^J a_j \Phi_j \chi_{E_j}. 
\end{equation} 
Then
\begin{equation}
\|S_{\wti \Phi}\|_{\cB(L^2(\bbR^n, \, \bbC^{m \times m}_{\rm HS}))} = \big\|\wti \Phi\big\|_{\infty,m^2}. 
\end{equation}
\end{lemma}
\begin{proof}
Without loss of generality we may assume in addition that the sets $E_j$ are pairwise disjoint, that 
$|E_j| > 0$, $1 \leq j \leq J$, $0 < |E_1| < \infty$, $|a_1| \geq |a_j| > 0$, $2 \leq j \leq J$, implying 
\begin{equation}
\big\|\wti \Phi\big\|_{\infty,m^2} = |a_1|. 
\end{equation} 
Since by assumption, $\|\Phi_1\|_{\cB(\bbC^{m \times m}_{\rm HS})} = 1$, there exists a sequence 
$\{u_{\ell}\}_{\ell \in \bbN} \subset \bbC^{m \times m}_{\rm HS}$, with $\|u_{\ell}\|_{\cB_2(\bbC^m)} = 1$, 
$\ell \in \bbN$, such that 
\begin{equation}
\lim_{\ell \to \infty} \|\Phi_1 u_{\ell}\|_{\cB_2(\bbC^m)} = 1. 
\end{equation} 
Introducing $f_{\ell} \in L^2(\bbR^n, \bbC^{m \times m}_{\rm HS})$, $\ell \in \bbN$, via
\begin{equation}
f_{\ell} = \big(|E_1|^{-1/2} u_{\ell} \chi_{E_1}\big)^{\vee}, \quad \ell \in \bbN,
\end{equation}
one infers,
\begin{align}
& \|f_{\ell}\|_{L^2(\bbR^n, \, \bbC^{m \times m}_{\rm HS})} = 
\big\|f^{\wedge}_{\ell}\big\|_{L^2(\bbR^n, \, \bbC^{m \times m}_{\rm HS})} = 
\big\||E_1|^{-1/2} u_{\ell} \chi_{E_1}\big\|_{L^2(\bbR^n, \, \bbC^{m \times m}_{\rm HS})}  \no \\
& \quad = |E_1|^{-1/2} \bigg(\int_{\bbR^n} 
\sum_{j,k=1}^m |(u_{\ell} \chi_{E_1}(x))_{j,k}|^2 \, d^nx\bigg)^{1/2}     \no \\
& \quad = |E_1|^{-1/2} \bigg(\int_{E_1} \sum_{j,k=1}^m |(u_{\ell})_{j,k}|^2 \, d^nx\bigg)^{1/2} = 1, 
\quad \ell \in \bbN,
\end{align}
and 
\begin{align}
& \|S_{\Phi} f_{\ell}\|^2_{L^2(\bbR^n, \, \bbC^{m \times m}_{\rm HS})} = 
\big\|(S_{\Phi} f_{\ell})^{\wedge}\big\|^2_{L^2(\bbR^n, \, \bbC^{m \times m}_{\rm HS})}   \no \\
& \quad = \bigg(\int_{\bbR^n} 
\sum_{j,k=1}^m \big|\big(\wti \Phi(x) |E_1|^{-1/2} u_{\ell} \chi_{E_1}(x)\big)_{j,k}\big|^2 \, d^nx\bigg)^{1/2} 
\no \\
& \quad = \bigg(|E_1|^{-1} \int_{\bbR^n} 
\sum_{j,k=1}^m |(a_1 \Phi_1 u_{\ell} \chi_{E_1}(x))_{j,k}|^2 \, d^nx\bigg)^{1/2}   \no \\
& \quad = \bigg(|E_1|^{-1} |a_1|^2 \int_{E_1} 
\sum_{j,k=1}^m |(\Phi_1 u_{\ell})_{j,k}|^2 \, d^nx\bigg)^{1/2}   \no \\
& \quad = |a_1| \|\Phi_1 u_{\ell}\|_{\cB_2(\bbC^m)} \underset{\ell \to \infty}{\longrightarrow} |a_1|. 
\end{align}
Thus,
\begin{equation}
\|S_{\wti \Phi}\|_{\cB(L^2(\bbR^n, \, \bbC^{m \times m}_{\rm HS}))} \geq |a_1| 
= \big\|\wti \Phi\big\|_{\infty,m^2}, 
\end{equation}
and Lemma \ref{lB.3} provides the converse inequality. 
\end{proof}

\begin{lemma} \lb{lB.5}
Assume \eqref{B.9}, then
\begin{equation}
\|S_{\Phi}\|_{\cB(L^2(\bbR^n, \, \bbC^{m \times m}_{\rm HS}))} = \|\Phi\|_{\infty,m^2}.
\end{equation}
\end{lemma}
\begin{proof}
Let $c_m \geq 1$ such that 
\begin{equation}
c_m^{-1} \max_{1 \leq j,k \leq m^2} |A_{j,k}| \leq \|A\|_{\cB(\bbC^{m \times m}_{\rm HS})} \leq 
c_m  \max_{1 \leq j,k \leq m^2} |A_{j,k}|, \quad A \in \bbC^{m^2 \times m^2}, 
\end{equation}
holds. Then, for (Lebesgue) a.e.\ $x \in \bbR^n$, 
\begin{equation}
|\Phi(x)_{j,k}| \leq \max_{1 \leq r,s \leq m^2} |\Phi(x)_{r,s}| \leq 
c_m \|\Phi(x)\|_{\cB(\bbC^{m \times m}_{\rm HS})} \leq c_m \|\Phi\|_{\infty,m^2}.  
\end{equation}
Thus, for each $j,k \in \{1,\dots, ,m^2\}$, there exists a sequence of simple functions 
$\Psi_{j,k,\ell} \colon \bbR^n \to \bbC$, $\ell \in \bbN$, such that for a.e.~$x \in \bbR^n$, 
\begin{equation}
|\Phi(x)_{j,k} - \Psi(x)_{j,k,\ell}| \leq 2^{(1/2) - \ell},    \lb{B.25}
\end{equation}
and
\begin{equation} 
 |\Psi(x)_{j,k,\ell}| \leq |\Phi(x)_{j,k}|. 
\end{equation}
Next, introduce $\Psi_{\ell} \colon \bbR^n \to \bbC^{m^2 \times m^2}$ via 
$(\Psi_{\ell}(x))_{j,k} = \Psi(x)_{j,k,\ell}$, $1 \leq j,k \leq m^2$, $x \in \bbR^n$. Then for 
a.e.\ $x \in \bbR^n$,
\begin{equation}
\|\Phi(x) - \Psi_{\ell}(x)\|_{\cB(\bbC^{m \times m}_{\rm HS})} \leq c_m 
\max_{1 \leq j,k \leq m^2} |\Phi(x)_{j,k} - \Psi(x)_{j,k,\ell}| \leq 2^{(1/2) - \ell} c_m.  \lb{B.27} 
\end{equation}
Combining Lemma \ref{lB.3} and \eqref{B.27} results in 
\begin{align}
\begin{split} 
& \|S_{\Phi} - S_{\Psi_{\ell}}\|_{\cB(L^2(\bbR^n, \, \bbC^{m \times m}_{\rm HS}))} 
= \|S_{(\Phi - \Psi_{\ell})}\|_{\cB(L^2(\bbR^n, \, \bbC^{m \times m}_{\rm HS}))}    \\
& \quad \leq \|\Phi - \Psi_{\ell}\|_{\infty,m^2} \leq 2^{(1/2) - \ell} c_m,  
\end{split} 
\end{align}
implying  
\begin{equation}
\big| \|S_{\Phi}\|_{\cB(L^2(\bbR^n, \, \bbC^{m \times m}_{\rm HS}))} 
- \|S_{\Psi_{\ell}}\|_{\cB(L^2(\bbR^n, \, \bbC^{m \times m}_{\rm HS}))}\big| 
\leq 2^{(1/2) - \ell} c_m.    \lb{B.29}
\end{equation}
Since $\Psi_{\ell}$ is a simple function, Lemma \ref{lB.4} implies 
\begin{equation}
\|S_{\Psi_{\ell}}\|_{\cB(L^2(\bbR^n, \, \bbC^{m \times m}_{\rm HS}))} 
= \|\Psi_{\ell}\|_{\infty,m^2}.  \lb{B.30}
\end{equation}
Employing \eqref{B.25} one obtains for a.e.\ $x \in \bbR^n$,
\begin{align}
\begin{split}
& \big|\|\Phi(x)\|_{\cB(\bbC^{m \times m}_{\rm HS})} 
- \|\Psi_{\ell}(x)\|_{\cB(\bbC^{m \times m}_{\rm HS})}\big|
\leq \|\Phi(x) - \Psi_{\ell}(x)\|_{\cB(\bbC^{m \times m}_{\rm HS})}   \\
& \quad \leq c_m \max_{1 \leq j,k \leq m^2} |\Phi(x)_{j,k} - \Psi_{\ell}(x)_{j,k}| \leq 2^{(1/2) - \ell} c_m, 
\end{split}
\end{align} 
implying 
\begin{equation}
\|\Phi\|_{\infty,m^2} = \lim_{\ell \to \infty} \|\Psi_{\ell}\|_{\infty,m^2}.    \lb{B.32}
\end{equation}
Combining \eqref{B.29}, \eqref{B.30}, and \eqref{B.32} finally yields 
\begin{equation}
\|S_{\Phi}\|_{\cB(L^2(\bbR^n, \, \bbC^{m \times m}_{\rm HS}))} = \|\Phi\|_{\infty,m^2}. 
\end{equation} 
\end{proof}

We note that Lemma \ref{lB.5} has been proven in \cite{GJR00} in the infinite-dimensional 
context. 

\begin{corollary} \lb{cB.6} 
Let $F \colon \bbR^n \to \bbC^{m \times m}$ be measurable and suppose that 
$\|F\|_{\infty,m} = {\rm ess.sup}_{x \in \bbR^n} \|F(x)\|_{\cB(\bbC^m)} < \infty$. Then, 
\begin{equation}
\|F(- i \nabla)\|_{\cB(L^2(\bbR^n, \, \bbC^{m \times m}_{\rm HS}))} = \|F\|_{\infty,m}.
\end{equation}
\end{corollary}
\begin{proof}
Given $A \in \bbC^{m \times m}$, let $M_A \in \cB(\bbC^{m \times m}_{\rm HS})$ be defined as in 
\eqref{B.2}, 
\begin{equation}
M_A(B) = BA, \quad B \in \bbC^{m \times m}_{\rm HS},
\end{equation}
and introduce $\Phi \colon \bbR^n \to \cB(\bbC^{m \times m}_{\rm HS})$ by
\begin{equation}
\Phi(x) = M_{F(x)}, \quad x \in \bbR^n.
\end{equation}
By Proposition \ref{pB.2}, 
\begin{equation}
\|\Phi(x)\|_{\cB(\bbC^{m \times m}_{\rm HS})} = \|F(x)\|_{\cB(\bbC^m)}, \quad x \in \bbR^n, 
\end{equation}
and hence by Lemma \ref{lB.5}, 
\begin{align}
& \|F(- i \nabla)\|_{\cB(L^2(\bbR^n, \, \bbC^{m \times m}_{\rm HS}))} 
= \|S_{\Phi}\|_{\cB(L^2(\bbR^n, \, \bbC^{m \times m}_{\rm HS}))} 
= {\rm ess.sup}_{x \in \bbR^n} \|\Phi(x)\|_{\cB(\bbC^{m \times m}_{\rm HS})}   \no \\
& \quad = {\rm ess.sup}_{x \in \bbR^n} \|F(x)\|_{\cB(\bbC^m)} = \|F\|_{\infty,m}. 
\end{align}
\end{proof}

\medskip

\noindent
{\bf Acknowledgments.} We are indebted to Loukas Grafakos, Alexander 
Sakhnovich, Lev Sakhnovich, Fedor Sukochev, and Yuri Tomilov for very helpful correspondence. Particular thanks are due to Loukas Grafakos for help with the 
proof of the multiplier equality \eqref{4.41A}, and to Yuri Tomilov for providing us 
with a most relevant list of references. We also thank both referees for a critical reading of our manuscript and for very helpful comments. 
 
 
\end{document}